\documentclass[10pt]{article}
\usepackage{amsmath,amssymb,amsthm,amsfonts,amsbsy}
\usepackage{graphicx}
\usepackage{verbatim,color}

\newcommand{\ee}{\mathbb{E}}
\renewcommand{\d}{{\rm d}}

\allowdisplaybreaks

\theoremstyle{theorem}

\newtheorem{lemma}{Lemma}[section]
\newtheorem{theorem}{Theorem}[section]

\newtheorem{proposition}{Proposition}[section]
\newtheorem{remark}{Remark}[section]
\numberwithin{equation}{section}

\begin{document}

\title{Numerical Approximation of Stochastic Time-Fractional Diffusion}
\author{Bangti Jin\thanks{Department of Computer Science, University College London, Gower Street, London WC1E 6BT, UK (\texttt{bangti.jin@gmail.com, b.jin@ucl.ac.uk})}
\and Yubin Yan\thanks{Department of Mathematics, University of Chester, Chester CHI 4BJ, UK (\texttt{y.yan@chester.ac.uk})}
\and Zhi Zhou\thanks{Department of Applied Mathematics, The Hong Kong Polytechnic University, Kowloon, Hong Kong (\texttt{zhizhou@polyu.edu.hk})}}

\maketitle

\begin{abstract}
We develop and analyze a numerical method for stochastic time-fractional diffusion driven by additive
fractionally integrated Gaussian noise. The model involves two nonlocal terms in time, i.e., a Caputo
fractional derivative of order $\alpha\in(0,1)$, and fractionally integrated Gaussian noise (with a
Riemann-Liouville fractional integral of order $\gamma \in[0,1]$ in the front). The numerical scheme
approximates the model in space by the Galerkin method with continuous piecewise linear finite elements
and in time by the classical Gr\"unwald-Letnikov method, and the noise by the $L^2$-projection. Sharp
strong and weak convergence rates are established, using suitable nonsmooth data error estimates for the
deterministic counterpart. Numerical results are presented to support the theoretical findings.
\end{abstract}

\section{Introduction}

In this work, we consider numerical methods for solving the following time-fractional
diffusion equation driven by fractionally integrated additive Gaussian noise, with   $0< \alpha <1, \, 0 \leq  \gamma  \leq 1$:
\begin{equation} \label{eqn:sfde}
   \partial_t^\alpha  u(t) +  A u(t) =  {_0I_t^\gamma} \dot{W}(t), \quad \forall  0 < t \leq T,\qquad  \mbox{with }
      u(0) =u_{0},
\end{equation}
where the notation $_0I_t^\gamma v(t) $ denotes the  Riemann-Liouville fractional
integral of order $\gamma>0$ of a function $v:[0,T]\to\mathbb{R}$ defined by
\begin{equation*}
 _{0}I_t^\gamma v(t)   = \frac{1}{\Gamma(\gamma)}   \int_{0}^{t} (t- s)^{\gamma -1}  v(s) \,\d s,
\end{equation*}
where $\Gamma(\cdot)$ denotes the Gamma function defined by $\Gamma(z)=\int_0^\infty s^{z-1}e^{-s}\d s$ (for $\Re z>0$),
with the convention $ _{0}I_t^0 v(t) = v(t)$. For $\gamma\in(-1,0)$, $_0I_t^\gamma v$ denotes the Riemann-Liouville
fractional derivative of order $-\gamma\in(0,1)$, defined by $_0I_t^\gamma v:=({_0I_t^{1+\gamma}}v(t))'$. The notation
$\partial_t^\alpha  v(t), \, 0 < \alpha <1, $ denotes the Caputo fractional derivative of order $\alpha$ defined by
\cite[p. 91]{KilbasSrivastavaTrujillo:2006}
\begin{equation*}
\partial_t^\alpha  v(t) = {_0I_t^{1-\alpha}}  v'(t).
\end{equation*}
In the model \eqref{eqn:sfde}, the operator $A$ denotes the negative Laplacian $ - \Delta$ with a zero Dirichlet
boundary condition in a convex polygonal domain $ D  \subset \mathbb{R}^{d}\, (d=1, 2, 3)$. Then the domain
$\mathcal{D}(A)$ of the operator $A$ is given by $ H_{0}^{1}( D ) \cap H^{2}(D)$. The noise $W(t)$ is given by
a Wiener process with a covariance operator $Q$ on a filtered probability space $( \Omega, \mathcal{F},
\mathbb{P},\{\mathcal{F}_t\}_{t\geq0})$, with $ \{\mathcal{F}_{t} \}_{t \geq0}$ being an increasing filtration
of $\sigma$-fields $\mathcal{F}_t\subset\mathcal{F}$, each of which contains all $(\mathcal{F},\mathbb{P})$-null
sets.  Let $\ee$ denote the expectation (with respect to $\mathbb{P}$). The function $u_{0}$ is an
$\mathcal{F}_0$-measurable random variable, and  belongs to $L^2(D)$ or its subspace. In order to ensure the
well-posedness of problem \eqref{eqn:sfde} \cite[pp. 1473--1474]{ChenKimKim:2015}, we assume the following condition:
\begin{equation}\label{ass:alpha}
\alpha\in(0,1),\quad \gamma\in[0,1]\quad \mbox{and}\quad \alpha+\gamma>1/2.
\end{equation}

The deterministic counterpart of the model \eqref{eqn:sfde}, commonly known as subdiffusion, has been extensively
studied in the literature over the last few decades \cite{JinLazarovZhou:2018review}, due to its numerous
applications in engineering, physics and biology \cite{MetzlerJeonCherstvy:2014}. The noise term
$W(t)$ in the model \eqref{eqn:sfde} is to describe random effects on transport of particles in
medium with memory or particles subject to sticking and trapping \cite{ChenKimKim:2015}. The
fractionally integrated noise $_0I_t^\gamma \dot W(t)$ reflects the fact that the internal energy depends also on the past
random effects. In recent years, stochastic fractional diffusion, e.g., the model \eqref{eqn:sfde},
has been very actively researched \cite{ChenKimKim:2015,Chen:2016,ChenHuNualart:2015,AnhLeonenkoRuiz:2017,
LiuRocknerSilva:2017}. Chen et al \cite{ChenKimKim:2015} studied the $L^{2}$ theory of \eqref{eqn:sfde}
in both divergence and non-divergence forms. Anh et al \cite{AnhLeonenkoRuiz:2017}
discussed sufficient conditions for a Gaussian solution (in the mean-square sense)
and derived temporal, spatial and spatiotemporal H\"older continuity of the solution.
Chen \cite{Chen:2016} analyzed moments, H\"{o}lder continuity and intermittency of the solution
for 1D nonlinear stochastic subdiffusion. Liu et al \cite{LiuRocknerSilva:2017} analyzed the existence
and uniqueness of the solution to \eqref{eqn:sfde} with fairly general quasi-linear elliptic operators.

To the best of our knowledge, there seems no work on the numerical analysis of the stochastic time-fractional PDEs
driven by fractionally integrated Gaussian noise, except in a few special cases. It is precisely this gap
that we aim at filling in the present work. Specifically, we develop a numerical scheme for problem
\eqref{eqn:sfde}, based on the standard Galerkin finite element method (FEM) with continuous linear finite elements
in space, the classical Gr\"unwald-Letnikov method (i.e., backward Euler convolution quadrature \cite{Lubich:1986,
LubichSloanThomee:1996,JinLazarovZhou:2016sisc}) in time and the $L^2$-projection of the noise, cf. \eqref{eqn:fullydiscrete}.
The scheme combines discretization techniques for subdiffusion \cite{JinLazarovZhou:2016sisc} and stochastic
heat equation \cite{Yan:2005}, and it is easy to implement. We prove nearly sharp strong and weak convergence rates
for the fully discrete approximation in Theorems \ref{thm:strong} and \ref{thm:weak}, respectively, which represent
the main theoretical contributions of the work.

The analysis employs an operator theoretic approach, which was first developed in the work \cite{Yan:2005}
for the stochastic heat equation and subsequently used in many works. In the analysis, one crucial ingredient
is certain nonsmooth data error estimates for solution operator associated with the deterministic inhomogeneous problem, i.e., the discrete
solution operators $\bar E_h(t)$ and $B_j$ in Section \ref{ssec:nonsmooth}. Due to the presence of the
fractional integral ${_0I_t^\gamma}$, such estimates differ greatly from that for subdiffusion, and
are still unavailable. We employ Laplace transform and generating function
\cite{LubichSloanThomee:1996} to derive requisite estimates. We refer interested readers to
\cite{CuestaLubichPalencia:2006,LinXu:2007,JinLazarovZhou:2013,JinLazarovZhou:2016sisc,McLeanMustapha:2015}
for related works on nonsmooth data estimates for deterministic subdiffusion; see also the survey
\cite{JinLazarovZhou:2018review} and the references therein. For the weak convergence, we employ a powerful
tool, i.e., Malliavin calculus, recently developed in \cite{AnderssonKruseLarsson:2016}. The
technique in \cite{AnderssonKruseLarsson:2016} relies on a new family of refined Sobolev-Malliavin spaces
that capture the temporal integrability of the Malliavin derivative, and a new Burkholder type inequality
in the dual norm of these Sobolev-Malliavin spaces. The challenge lies in
deriving the error estimate in the dual norm of refined Sobolev-Malliavin spaces.

\begin{table}[hbt!]
\centering
  \caption{Convergence rates for the numerical scheme with $u_0=0$ and trace class noise. \label{tab:conv}}
  \begin{tabular}{l|cc}
  \hline
  $(\alpha,\gamma)$ & strong  \\
  \hline
  $(1,0)$ & $O(h+\tau^{\frac12-\epsilon})$ \cite{Yan:2005} \\
  \hline
  $\gamma<1/2$ & $O(h^{2-\frac{1-2\gamma}{\alpha}-\epsilon}+\tau^{\min(1,\alpha+\gamma-\frac{1}{2}-\epsilon)})$ \\
  $\gamma>1/2$ & $O(h^2+\tau^{\min(1,\alpha+\gamma-\frac{1}{2}-\epsilon)})$ \\
  \hline
  \end{tabular}
  \quad
  \begin{tabular}{l|ccc}
  \hline
  $(\alpha,\gamma)$ & weak \\
  \hline
  $(1,0)$ &  $O(h^2+\tau^{1-\epsilon})$ \\
  \hline
  $\gamma<\frac{1-\alpha}{2}$ & $O(h^{4-\frac{2(1-2\gamma)}{\alpha}-\epsilon}+\tau^{\min(1,\alpha+\gamma-\epsilon)})$\\
  $\gamma>\frac{1-\alpha}{2}$ & $O(h^2+\tau^{\min(1,\alpha+\gamma-\epsilon)})$\\
  \hline
  \end{tabular}
\end{table}

Theorems \ref{thm:strong} and \ref{thm:weak} indicate that the fractionally integrated Gaussian noise
$_0I_t^\gamma \dot W(t) $ induces convergence behaviors substantially different from that of stochastic
diffusion. In particular, the fractional order $\gamma$ can exert strong influence on both strong and weak
convergence rates: dependent of the $\gamma$ value, with $h$ and $\tau$ being the mesh size and time step
size, respectively, the spatial convergence rate may reach $O(h^2)$ and the temporal convergence rate
$O(\tau)$ in both strong and weak sense, and the usual dichotomy of the weak temporal rate being twice the strong
one is generally not valid; See Table \ref{tab:conv} for convergence rates when the noise $W(t)$ belongs to
trace class, where the results for stochastic diffusion (i.e., $(\alpha,\gamma)=(1,0)$) are also given
for comparison. In the table, $\epsilon$ is an arbitrarily small positive constant. Further, the results for
stochastic diffusion are recovered upon letting $\alpha\to1^-$ and $\gamma\to0^+$. These theoretical
findings are fully supported by the extensive numerical experiments in Section \ref{sec:numeric}.

To the best of our knowledge, there is no work on the numerical analysis of the general model \eqref{eqn:sfde},
except some special cases, which we describe next. First, the model \eqref{eqn:sfde} represents the
fractional analogue of the classical heat equation (but with a nonstandard noise term), and recovers the latter model for the
special choice $(\alpha=1,\gamma=0)$. Thus, naturally our results generalize the corresponding results
for stochastic heat equation; see Table \ref{tab:conv} for the case of trace class noise.
The literature on stochastic heat equation is vast. See, e.g., \cite{AllenZhang:1998,DuZhang:2002,Yan:2005} for
strong convergence, and, e.g., \cite{DebusschePrintems:2009,AnderssonLarsson:2016,BrehierHairerStuart:2018},
for weak convergence, and interested readers are also referred to the surveys \cite{JentzenKloeden:2009,
Kruse:2014} and references therein for further pointers to the vast literature. Second, the stochastic fractional
model \eqref{eqn:sfde} was studied earlier for the case $\gamma=1-\alpha$ \cite{LiYang:2017}, where the strong
convergence of a discontinuous Galerkin method in time was analyzed; see also \cite{GunzburgerLiWang2:2017} for a
related fractional-order model with white noise.

The  rest of the paper is organized as follows. In Section \ref{sec:prelim}, we give preliminaries on
Wiener process and Malliavin calculus. In Section \ref{sec:fully}, we describe the numerical scheme,
and in Section \ref{sec:nonsmooth}, we derive crucial nonsmooth data error estimates for deterministic
subdiffusion. The strong and weak error estimates for approximations are given in Section \ref{sec:err-stochastic}.
In Section \ref{sec:numeric}, we discuss the implementation of the noise,
and present numerical results to support the theoretical analysis. In Appendix \ref{app:reg}, we present
some regularity results. Throughout, the notation $c$, with or without a subscript, denotes a generic
constant, which may differ at each occurrence, but it is always independent of the mesh size $h$ and
the time step size $\tau$. Further, $\epsilon>0$ is always a small positive constant.

\section{Preliminaries}\label{sec:prelim}
In this section, we collect preliminary facts on Wiener process and Malliavin calculus.
\subsection{Wiener process}

Let $(U,\|\cdot\|_U,\langle\cdot,\cdot\rangle_U)$ and $(V,\|\cdot\|_V,\langle\cdot,\cdot\rangle_V)$ be separable Hilbert
spaces. Let $L(U;V)$ be the Banach space of all bounded linear operators $U\to V$, and we denote $L(U)=L(U;U)$.
$\mathcal{L}_2(U;V)\subset L(U;V)$ denotes the subspace of all Hilbert-Schmidt operators, with norms and
inner products respectively given by
\begin{equation*}
  \|T\|_{\mathcal{L}_2(U;V)}^2  = \sum_{j\in\mathbb{N}}\|Tu_j\|_V^2,\quad \langle S,T\rangle = \sum_{j\in\mathbb{N}}\langle Su_j,Tu_j\rangle_V.
\end{equation*}
Both are independent of specific choice of orthonormal basis $\{u_j\}_{j\in\mathbb{N}}$.

Let $H= L^{2}(D)$ with the norm $\| \cdot \|$ and the inner product $(\cdot, \cdot )$. A Wiener process $W(t)$
with covariance $Q$ may be characterized as follows. Let $Q$ be a bounded, linear, selfadjoint, positive definite
operator on $H$, with the pairs of eigenvalue and eigenfunction denoted by $\{(\gamma_{\ell},e_\ell)\}_{\ell=1
}^\infty$. Let $H_0=Q^\frac{1}{2}H$ be the Hilbert space endowed with the inner product $\langle u,v\rangle_{H_0}
=(Q^{-\frac12}u,Q^{-\frac12}v)$. Let $\{\beta_{\ell} (t)\}_{\ell=1}^\infty$ be a sequence of real-valued
independently and identically distributed (i.i.d.) Brownian motions. Then the series
\begin{equation}\label{eqn:Wiener}
W(t) =\sum_{\ell=1}^{\infty} \gamma_{\ell}^\frac{1}{2} e_{\ell} \beta_{\ell}(t),
\end{equation}
is a Wiener process with covariance operator $Q$.
If $Q$ is of trace class, i.e., $\sum_{\ell=1}^\infty\gamma_\ell<\infty$, then $W(t)$ is an $H$-valued process.
If $Q$ is not in trace class, e.g.,  $Q=I$, then $W(t)$ does not belong to  $H$,
in which case $W(t)$ is called a cylindrical Wiener process \cite[Chapter 4]{DaPratoZabczyk:2014}.

The notation $\mathcal{L}_{2}^{0} =\mathcal{L}_2(H_0;H)$
denotes the space of Hilbert-Schmidt operators from $H_0$ to $H$, i.e.,
\begin{equation*}
  \mathcal{L}_2^0 = \Big\{\Phi\in L(H): \ \sum_{\ell=1}^\infty\|\Phi Q^\frac{1}{2}e_\ell\|^2<\infty\Big\},
\end{equation*}
with the norm $\|\cdot\|_{\mathcal{L}_2^0}$ defined by
\begin{equation*}
  \|\Phi\|_{\mathcal{L}_2^0} = \Big(\sum_{\ell=1}^\infty\|\Phi Q^\frac12e_\ell\|^2\Big)^\frac12,
\end{equation*}
where $\{e_\ell\}_{\ell=1}^\infty$ is an orthonormal basis of $H$. This definition is independent of the choice of the basis.
For any $ \Phi \in L^{2}(0, T; \mathcal{L}_{2}^{0})$, $\int_{0}^{t}\Phi (s)  \d W(s)$
is well defined in the sense of stochastic integral \cite[p. 95]{DaPratoZabczyk:2014}.

For any $p\geq 1$, we define the space of $H$-valued $p$-integrable random variables by
\begin{equation*}
L^p(\Omega; H) = \Big \{ v: \, \ee \| v \|^p = \int_{\Omega} \| v(\omega) \|^p \, \d \mathbb{P} (\omega) < \infty \Big \},
\end{equation*}
with norm $ \| v \|_{L^p(\Omega; H)}= (\ee \| v\|^p)^\frac1p $. Similarly, we define
the space $L^{p}(\Omega; V)$, for a Banach space $V$.

\subsection{Malliavin calculus}
In this part, we recall some concepts related to Malliavin derivatives of $H$-valued random
variables. More details can be found in, e.g., \cite{AnderssonKruseLarsson:2016,Kruse:2014}.
Let $\mathcal{G}_{p}^{\infty}(\mathbb{R}^{n}; \mathbb{R})$ be the space of all infinitely many times
G\^ateaux differentiable mappings $\phi: \mathbb{R}^{n} \to \mathbb{R}$ such that $\phi$ and all its
derivatives satisfy a polynomial bound. Let $\mathcal{B}(H;\mathbb{R})$ denote the Banach space of all bilinear
mappings $b: H \times H \to \mathbb{R}$ equipped with the norm
\begin{equation*}
\| b \|_{\mathcal{B}(H;\mathbb{R})} = \sup_{0\neq u_1, u_2 \in H} \frac{ | b \cdot (u_{1}, u_{2}) |}{\| u_{1} \| \|u_{2}\|}.
\end{equation*}
For any $\ell\geq 2$, let $\Phi \in \mathcal{G}_{p}^{2,\ell}(H;\mathbb{R})$, with
\begin{equation}
\mathcal{G}_{p}^{2,\ell}(H;\mathbb{R})
=\Big\{ \Phi: \, H \to \mathbb{R}, \; | \Phi |_{\mathcal{G}_{p}^{2,\ell} (H;\mathbb{R})}
=\sup_{u \in H} \frac{ \| \Phi^{(2)} (u) \|_{\mathcal{B}(H;\mathbb{R})}}{(1+ \| u \|_{H}^{\ell-2})} <\infty\Big\},
\end{equation}
where $\Phi^{(2)} (u) \in \mathcal{B}(H; \mathbb{R})$ denotes the second-order G\^ateaux derivative of
$\Phi\in\mathcal{G}_p^{2,\ell}(H;\mathbb{R})$ at $u \in H$.

For $q \in [2, \infty]$, $S^{q} (\mathbb{R})$ denotes the
class of smooth cylindrical random variables of the form
\[
F= f \Big ( \int_{0}^{T} \phi_{1}(s) \, \d W(s), \dots, \int_{0}^{T} \phi_{n}(s) \, \d W(s) \Big ),
\]
where
$f \in \mathcal{G}_{p}^{\infty} (\mathbb{R}^{n};\mathbb{R})$ and $\{\phi_{k}\}_{k=1}^{n} \subset L^{q} (0, T;
\mathcal{L}_2(H;\mathbb{R}))$, $n \in \mathbb{N}$. (Recall that the space $\mathcal{L}_2(H;\mathbb{R})$ is defined by
$\mathcal{L}_2(H;\mathbb{R})=\big \{ \Phi \in L(H; \mathbb{R}): \sum_{\ell=1}^{\infty} |  \Phi Q^\frac{1}{2}
e_{\ell} |^{2}_{\mathbb{R}} < \infty  \big \}$, where $L(H; \mathbb{R})$ denotes the Hilbert space of all bounded
operators from $H$ to $\mathbb{R}$ and $| \cdot |_{\mathbb{R}}$ denotes the Euclidean norm in $\mathbb{R}$.) For $F \in
S^{q}(\mathbb{R})$, we define the Malliavin derivative by
\[
D F (\sigma) = \sum_{j=1}^{n} \partial_{j} f  \Big ( \int_{0}^{T}\phi_{1}(s) \, \d W(s), \dots, \int_{0}^{T}\phi_{n}(s) \, \d W(s) \Big )
\otimes \phi_{j} (\sigma),\quad \sigma\in [0,T].
\]
Note that $DF(\sigma)$ is an $H_0$-valued stochastic process.

Next, we recall the Malliavin derivative for $H$-valued random variables. Let $S^q(H)$ be the
space of all $H$-valued random variables of the form $Y= \sum_{i=1}^mv_{i} \otimes F_{i}$ with
$\{v_i\}_{i=1}^{m} \subset H $, $\{F_i\}_{i=1}^m \subset S^q(\mathbb{R})$, $m \in \mathbb{N}$,
where $\otimes$ denotes the tensor product. Then the Malliavin derivative of $Y \in S^{q}(H)$ is defined by
\[
D Y (\sigma) = \sum_{i=1}^{m} v_i \otimes D F_i(\sigma).
\]
Since  $D F_{i} (\sigma)$ is an $H_0$-valued stochastic process, $D Y (\sigma)$ is an $H \otimes
H_0 = \mathcal{L}_2^0$-valued process.

For $p \in [2, \infty),q\in [2,\infty]$, $ S^q(H) \subset L^p(\Omega;H)$ is dense
\cite[Lemma 3.1]{AnderssonKruseLarsson:2016}. Further, the operator $D: S^q(H) \to L^p(\Omega,
L^q(0, T;\mathcal{L}_2^0))$ is closable \cite[Lemma 3.2]{AnderssonKruseLarsson:2016}.
Let $M^{1, p, q} (H)$ be the closure of $S^{q}(H)$ with respect to the norm
\[
\|Y\|_{M^{1,p,q}(H)}= \big(\| Y\|_{L^p(\Omega;H)}^p+\| D Y \|_{L^p(\Omega;L^q(0, T;\mathcal{L}_2^0))}^p\big)^\frac1p.
\]
Then $M^{1, p, q}(H)$ are Banach spaces, densely embedded into $L^{2}(\Omega;H)$, and
$M^{1, p, q}(H) \subset L^{2}(\Omega;H) \subset M^{1, p, q}(H)^{*}$ is a Gel'fand triple.
Further, we denote $M^{1, p}(H) = M^{1, p, p}(H)$ and $M^{1, p}(H)^{*} = M^{1, p, p}(H)^{*}$

We shall use frequently Burkholder's inequality (\cite[Lemma 7.2]{DaPratoZabczyk:2014} and
\cite[Lemma 3.5]{AnderssonKruseLarsson:2016}). For any exponent $p\geq 1$, $p'\geq1$ denotes its
conjugate exponent, i.e., $p^{-1}+{p'}^{-1}=1$.
\begin{lemma} \label{lem:burkholder}
For $ p \geq 2$, let $(\Phi (t))_{t \in [0, T]}$ be a predictable and $\mathcal{L}_2^0$-valued stochastic process
such that $\|\Phi \|_{L^{p}(\Omega;L^{2}(0, T;\mathcal{L}_2^0))}  <\infty$. Then there hold
\begin{align}
\| \int_{0}^{T} \Phi (t) \d W(t) \|_{L^p(\Omega; H)}
&\leq c \| \Phi \|_{L^p(\Omega; L^{2}(0, T;\mathcal{L}_2^0))},\label{burkholder0}\\
 \| \int_{0}^{t} \Phi (t)\d W(t) \|_{M^{1, p}(H)^{*}} &\leq c\|\Phi \|_{L^{p'} (\Omega;L^{p'}(0,T;\mathcal{L}_2^0))}. \label{burkholder}
\end{align}
\end{lemma}

Last, we recall one result on the chain rule for Malliavin derivative \cite[Lemma 3.3]{AnderssonKruseLarsson:2016}.
\begin{lemma}\label{lem:bd-Malliavin}
Let $U, V$ be two separable Hilbert spaces and $\gamma \in C^{1}(U;V)$ satisfy for $u \in U$
\begin{align*}
 \| \gamma (u) \|_{V} \leq c( 1 + \| u \|_{U}^{1+ r}) \quad \mbox{and}\quad
\| \gamma^{\prime} (u) \|_{L(U;V)} \leq c( 1 + \| u \|_{U}^{r}),
\end{align*}
for some $r \geq 0$. Then for $ u \in M^{1, (1+r)p, q} (U)$ with $p >1$ and $q \geq 2$,
$ \gamma (u) \in  M^{1, p, q} (V)$ and
\begin{equation*}
\|\gamma (u) \|_{M^{1, p, q} (V)}
\leq c( 1+ \| u \|_{M^{1, (1+r)p, q} (U)}^{1+r})
\quad\mbox{and}\quad
D [ \gamma(u)](\sigma)
= \gamma^{\prime} (u) D [u] (\sigma), \; \sigma \in [0, T].
\end{equation*}
\end{lemma}

\section{Numerical scheme}\label{sec:fully}
Now we develop a numerical scheme for problem \eqref{eqn:sfde} based on the Galerkin FEM with conforming piecewise linear
FEM in space, Gr\"{u}nwald-Letnikov formula in time, and $L^2$-projection of the noise $W(t)$. Let  $\mathcal{T}_h$
be a shape regular quasi-uniform triangulation of the domain $D$, and $X_{h}\subset H_0^1(D)$ be the space of
continuous piecewise linear functions on $\mathcal{T}_h$. On the FEM space $X_h$, we define the
$L^2(D)$-projection  $P_{h}: H \to X_{h}$ by
\begin{equation*}
( P_{h} v , \chi) =  ( v, \chi ), \quad \, \forall v\in H, \chi \in X_{h}.
\end{equation*}
Further, let $A_{h}:X_{h} \to X_{h}$ be  the discrete analogue of the negative Laplacian $A$, defined by
\begin{equation*}
  ( A_{h} v, \chi) = a(v, \chi),\quad \forall v, \chi \in X_{h},
\end{equation*}
where $a(v,\chi)=(\nabla v,\nabla\chi)$ is the bilinear form associated with $A$.
Then the semidiscrete Galerkin  FEM scheme reads: Given $u_h(0)=P_hu_0$, find $u_h(t)\in X_h$ such that
\begin{equation}\label{eqn:semi}
   \partial_t^\alpha u_h(t) +  A_h u_h(t) =  {_0I_t^\gamma} P_h\dot{W}(t), \quad \forall  0 < t \leq T.
\end{equation}

For the time discretization, let $t_n=n\tau$, $n=0,\ldots,N$, be a uniform partition of the interval $[0, T]$
and  $\tau=T/N $ the time step size. We approximate the Riemann-Liouville fractional integral / derivative
$_0I_t^\gamma v(t_n)$ by Gr\"{u}nwald-Letnikov formula (with $v^k=v(t_k)$)
\begin{equation}\label{eqn:GL}
  _0I_t^\gamma v(t_n) \approx \tau^\gamma\sum_{k=0}^nb_{n-k}^{(-\gamma)}v^k,
\end{equation}
where the weights $b_{j}^{(-\gamma)}$ are generated by power series expansion (with $\delta(\zeta)= 1-\zeta$):
\begin{equation*}
  \delta(\zeta)^{-\gamma} = \sum_{j=0}^{\infty} b_j^{(-\gamma)} \zeta^{j}.
\end{equation*}
The coefficients $b_j^{(-\gamma)}$ can be computed efficiently via a recursion formula.
Since $\partial_t^\alpha u = {_0I_t^{-\alpha}}(u-u(0))$ \cite[p. 91]{KilbasSrivastavaTrujillo:2006},
upon letting $f^{0} =0$ and
\begin{equation*}
f^{k} = \tau^{-1}P_h\Delta W^k, \quad \mbox{ with } \Delta W^k=W(t_{k})- W(t_{k-1}),\quad k=1, 2, \dots, N,
\end{equation*}
the numerical scheme for problem \eqref{eqn:sfde} reads: find $U^{n}\in X_h$ such that
\begin{equation}\label{eqn:fullydiscrete}
 \tau^{-\alpha}  \sum_{k=0}^{n}b_{n-k}^{(\alpha)}(U^{k}-U^0)
 + A_{h} U^{n} = \tau^{\gamma}\sum_{k=0}^{n}b_{n-k}^{(-\gamma)} f^{k},\quad \; n=1, 2, \dots, N,
\end{equation}
with the initial data $ U^{0} = u_h(0)$. We refer to Section \ref{ssec:implement} for implementation details.

\section{Nonsmooth data estimates}\label{sec:nonsmooth}
In this part we prove certain nonsmooth data error estimates.
\subsection{Solution representations} First we give the solution representations, which are useful
for nonsmooth error analysis below. Problem \eqref{eqn:sfde} admits a unique mild solution of the form
\begin{equation} \label{eqn:mild}
u(t) = E(t) u_{0} + \int_{0}^{t} \bar E(t-s)   \, \d W(s),
\end{equation}
where the solution operators $E$ and $\bar E$ are respectively defined by
\begin{equation} \label{eqn:solop}
E(t) = \frac{1}{2\pi\rm i} \int_\Gamma e^{zt}z^{\alpha-1}(z^\alpha+A)^{-1} \d z \quad\mbox{and}\quad
\bar E(t)= \frac{1}{2\pi\rm i} \int_\Gamma e^{zt}z^{-\gamma}(z^\alpha+A)^{-1} \d z.
\end{equation}
Here $\Gamma$ is a line in the complex plane $\mathbb{C}$ with $ \Re z = a >0$ for some $a>0$. One can
deform $\Gamma$ to $\Gamma_{\theta,\delta}:=\{z\in\mathbb{C}:z=re^{\pm\mathrm{i}\theta},\ r\geq \delta\}
\cup\{z\in\mathbb{C}: z=\delta e^{\mathrm{i}\varphi}, \ |\varphi|\leq \theta\}$ for some $\theta>\pi/2$.
The representation \eqref{eqn:mild} can be derived from Laplace transform as follows. Let
$g:\mathbb{R}_+\mapsto H$ be subexponential, i.e., for any $\epsilon>0$, the
function $t\to g(t)e^{-\epsilon t}$ belongs to $L^1(\mathbb{R}_+,H)$. We define Laplace transform
$\widehat{g}:\mathbb{C}_+\mapsto H$ by $\widehat g(z) = \int_0^\infty g(t)e^{-zt}\d t$,
where $\mathbb{C}_{+} = \{ z \in \mathbb{C}, \, \Re z>0 \}$. Then by applying Laplace transform to the
following deterministic problem
\begin{equation*}
\partial_t^\alpha u(t) +  A u(t) = \,  {_0I_t^\gamma} f(t),
\end{equation*}
with $u(0)=u_0$, and using the identities $\widehat {\partial_t^\alpha u}(z)=z^\alpha
\widehat u - z^{\alpha-1}u_0$ \cite[p. 98, Lemma 2.24]{KilbasSrivastavaTrujillo:2006} and
$\widehat{{_0I_t^\gamma}f}(z)=z^{-\gamma}\widehat f(z)$ (for $\gamma>0$)
\cite[p. 84, Lemma 2.14]{KilbasSrivastavaTrujillo:2006}, we obtain
\begin{equation*}
z^{\alpha} \widehat{u} (z) - z^{\alpha -1} u_0 +A \widehat{u} (z)
= z^{-\gamma} \widehat{f}(z),
\end{equation*}
i.e., $\widehat{u}(z) = (z^{\alpha} +A)^{-1} z^{\alpha-1} u_0
+ (z^{\alpha} +A)^{-1} z^{-\gamma} \widehat{f}(z).$ Then by inverse Laplace transform, we obtain \eqref{eqn:mild}.

The analysis below relies on smoothing properties of $E(t)$ and $\bar E(t)$ in the space $\dot H^r(D)$. For
any $r \in \mathbb{R}$, let the space $\dot{H}^{r}(D) = \mathcal{D}(A^\frac{r}{2})$ with the norm given by
$|v|_{r} = \| A^\frac{r}{2} v \|$. We use extensively the following estimates on $E(t)$ and $\bar E(t)$ below.
\begin{lemma} \label{lem:solop}
For $p,q\in\mathbb{R}$ with $ 0\leq p-q \leq 2$, there hold
\begin{align*}
|E(t) v |_p\leq c t^{-\alpha \frac{p-q}{2}}|v|_q\quad\mbox{and}\quad
|\bar E(t)v|_{p} + t|\bar E'(t)v|_{p}  \leq ct^{-\alpha \frac{p-q}{2} + (\alpha + \gamma -1)}|v|_{q}.
\end{align*}
\end{lemma}
\begin{proof}
Recall the resolvent estimate \cite[Example 3.7.5 and Theorem 3.7.11]{ArendtBattyHieber:2011}
\begin{equation}\label{eqn:resolv}
  \| (z+A)^{-1} \|\le c_\phi |z|^{-1},  \quad \forall z \in \Sigma_{\phi}\equiv \{0\neq z\in\mathbb{C}:|\arg(z)|\leq \phi\},
  \,\,\,\forall\,\phi\in(0,\pi).
\end{equation}
Then simple computation gives $\|z^{\alpha-1}A^r(z^\alpha+A)^{-1}\|\leq c|z|^{r\alpha-1}$ for $z\in \Sigma_\theta$ and $r\in[0,1]$.
Thus, taking $\delta=t^{-1}$ in the contour $\Gamma_{\theta,\delta}$ leads to
\begin{align*}
  |E(t)v|_p  &\leq \int_{\Gamma_{\theta,\delta}}e^{\Re zt}\|A^\frac{p-q}{2}(z^\alpha+A)\| |\d z|\|A^\frac{q}{2}v\| \\
     & \leq c|v|_q\Big(\int_{t^{-1}}^\infty e^{t\rho\cos\theta} \rho^{\frac{p-q}{2}\alpha-1}\d \rho
     + ct^{-\frac{p-q}{2}\alpha}\int_{-\theta}^\theta \d \phi\Big)\leq ct^{-\frac{p-q}{2}\alpha}|v|_q.
\end{align*}
This shows the estimate on $E(t)$, and the other follows similarly.
\end{proof}

Likewise, the semidiscrete solution $u_{h}(t)\in X_{h}$ to problem \eqref{eqn:semi} is represented by
\begin{equation*}
u_{h}(t) = E_{h}(t) P_{h} u_{0} + \int_{0}^{t} \bar E_{h}(t-s)  P_{h} \d W(s),
\end{equation*}
with the discrete analogues of $E(t)$ and $\bar E(t)$, defined by
\begin{equation} \label{eqn:solop-semi}
E_h(t) = \frac{1}{2\pi\rm i} \int_{\Gamma_{\theta,\delta}} e^{zt}z^{\alpha-1}(z^\alpha+A_h)^{-1} \d z \quad\mbox{and}\quad
\bar E_h(t)= \frac{1}{2\pi\rm i} \int_{\Gamma_{\theta,\delta}} e^{zt}z^{-\gamma}(z^\alpha+A_h)^{-1} \d z.
\end{equation}

Next, we give a representation of the solution $U^n$ to the scheme \eqref{eqn:fullydiscrete}. For a given sequence
$\{f^n\}_{n=0}^\infty$, the generating function is given by $\widetilde f(\zeta)$,
i.e., $\widetilde f(\zeta)=\sum_{n=0}^\infty f^n\zeta^n. $ Next we introduce
 operators $B_j$ by
\begin{equation} \label{eqn:B}
  \widetilde{B} (\zeta) = \sum_{j=0}^{\infty} B_{j} \zeta^{j} \quad  \mbox{with } \widetilde{B} (\zeta)
  = 1 + \zeta\big ( \tau^{-\alpha} \delta (\zeta)^{\alpha} +A_h \big )^{-1} \tau^{\gamma-1} \delta(\zeta)^{-\gamma}.
\end{equation}

\begin{proposition}
The solution $U^n$ to the scheme \eqref{eqn:fullydiscrete} is given by
\begin{equation}\label{eqn:fullydiscrete1}
 U^{n} = U_h^{n} + \tau \sum_{k=1}^{n} B_{n-(k-1)} f^{k}, \quad \mbox{with} \; n=1,2,\ldots,
\end{equation}
where $U_h^n$ is the fully discrete solution to the homogeneous problem of \eqref{eqn:fullydiscrete}.
\end{proposition}
\begin{proof}
We split $U^n$ into $U^{n} = U_h^{n} + U_i^{n},$ where $U_h^{n}$ and $U_i^n$
are the solutions to the homogeneous and inhomogeneous problems of \eqref{eqn:fullydiscrete}, respectively, where $U_i^n$ satisfies
\begin{equation*}
  \tau^{-\alpha}  \sum_{k=0}^{n}b_{n-k}^{(\alpha)} U_i^{k}
 + A_{h} U_i^{n} = \tau^{\gamma}  \sum_{k=0}^{n}b_{n-k}^{(-\gamma)} f^{k},\quad n=1,2,\ldots,
\end{equation*}
with $U_i^0=0$. Multiplying both sides with $\xi^n$ and summing over $n$ from $1$ to $\infty$ yield
\begin{equation*}
\tau^{-\alpha}\sum_{n=1}^{\infty} \Big (
 \sum_{k=0}^{n}b_{n-k}^{(\alpha)} U_i^{k} \Big ) \zeta^{n}+
\sum_{k=1}^\infty(A_hU_i^{n})\zeta^{n}=\tau^{\gamma}  \sum_{n=1}^{\infty} \Big (
\sum_{k=0}^nb_{n-k}^{(-\gamma)} f^{k} \Big ) \zeta^{n}.
\end{equation*}
Since $U_i^0=0$ and $f^0=0$, by discrete convolution rule and the definitions of
$\widetilde{U}_i(\zeta)$ and $\widetilde f(\zeta)$,
\begin{equation*}
\sum_{n=1}^{\infty} \Big (  \sum_{k=0}^nb_{n-k}^{(\alpha)} U_i^{k} \Big ) \zeta^{n}
=\delta (\zeta)^{\alpha} \widetilde{U}_i (\zeta)
\quad \mbox{and}\quad \sum_{n=1}^{\infty} \Big (
  \sum_{k=0}^{n}b_{n-k}^{(-\gamma)} f^{k} \Big ) \zeta^{n}
= \delta(\zeta)^{-\gamma} \widetilde f(\zeta),
\end{equation*}
from which it directly follows
\begin{equation*}
\widetilde{U}_i(\zeta) = \big(\tau^{-\alpha}\delta(\zeta)^{\alpha}+A_{h} \big )^{-1}  \tau^{\gamma}\delta(\zeta)^{-\gamma}  \widetilde f(\zeta).
\end{equation*}
By the defining relation \eqref{eqn:B} of $\widetilde B$ and noting $f^0=0$, we have
\begin{align}
\widetilde{U}_i (\zeta) & = \tau \frac{\widetilde{B}(\zeta) -1}{\zeta}  \widetilde f(\zeta)
 = \tau \sum_{n=1}^{\infty} \Big ( \sum_{k=0}^{n-1} B_{n-k} f^{k+1} \Big ) \zeta^{n}
=  \tau \sum_{n=1}^{\infty} \Big ( \sum_{k=1}^{n} B_{n-(k-1)} f^{k} \Big ) \zeta^{n}, \notag
\end{align}
which implies directly the desired relation.
\end{proof}

The next result holds for the solution $U_h^n$ to the homogeneous problem \cite[Theorem 3.5]{JinLazarovZhou:2016sisc}.
\begin{lemma} \label{lem:error-homo}
Let $u(t_{n})$ and $U_h^{n}$ be the solution of homogeneous problem and its fully discrete approximation by the
scheme \eqref{eqn:fullydiscrete}, respectively. Then there holds for $0\leq q\leq 2$
\begin{equation*}
\|  U_h^{n}  - u(t_{n} ) \| \leq c(\tau  t_{n}^{-1+\frac{q}{2}\alpha}  + h^{2}   t_{n}^{\frac{q-2}{2}\alpha})| u_{0}|_{q}.
\end{equation*}
\end{lemma}

\subsection{Nonsmooth data estimates}\label{ssec:nonsmooth}

Now we derive some important error estimates for $\bar E_h$ and $B_j$, which are crucial for the error analysis of the scheme \eqref{eqn:fullydiscrete}.
First, we give spatial discretization errors. On the space $X_h$, for any $r\in \mathbb{R}$, we define
the norm $\||\chi\||_r= \|A_h^\frac{r}{2} \chi\|_{L^2(D)}$, which is the discrete analogue of the norm
$|\cdot|_r$. Clearly, $\||\cdot\||_0$ coincides the usual $L^2(D)$-norm. Further, on quasiuniform
triangulations $\mathcal{T}_h$, for $g\in \dot H^r(D)$ with $0\leq r\leq 1$, there holds
\begin{equation*}
  \||P_hg\||_r\leq c|g|_r.
\end{equation*}
In fact, the case $r=0$ follows by the $L^2(D)$-stability of $P_h$, and the case $r=1$ by the $H^1(D)$-stability of
$P_h$. The case $r\in(0,1)$ follows by interpolation. Further, the following bound holds
\begin{equation}\label{eqn:AAh}
\| A_{h}^{-\frac{s}{2}} P_{h} A^{\frac{s}{2}} \| \leq c, \quad 0 \leq s \leq 1. 
\end{equation}
Actually, the case $s=0$ is trivial. Meanwhile, by \cite[(3.15)]{Thomee:2006}, $\|A_h^{-\frac12}P_hv\|\leq |v|_{-1}$
for all $v\in \dot H^{-1}(D)$. Hence, $\|A_h^{-1/2}P_hA^{1/2}\|\leq 1$, and by interpolation, the bound
\eqref{eqn:AAh} follows.

The operator $\bar E_h(t)$ satisfies the following smoothing property, similar to Lemma \ref{lem:solop}.
The proof follows from the resolvent estimate for $A_h$ \cite[p. 93]{Thomee:2006}:
\begin{equation*}
  \| (z+A_h)^{-1} \|\le c_\phi |z|^{-1},  \quad \forall z \in \Sigma_{\phi},
  \,\,\,\forall\,\phi\in(0,\pi).
\end{equation*}
\begin{lemma} \label{lem:solop-discrete}
For $p,q\in\mathbb{R}$ with $ 0\leq p-q \leq 2$, there holds
\begin{align*}
\||\bar E_h(t)\chi\||_{p} & \leq ct^{-\alpha \frac{p-q}{2} + (\alpha + \gamma -1)}\||\chi\||_{q}\quad\forall\chi\in S_h.
\end{align*}
\end{lemma}

The next lemma gives an error estimate on $\bar E_h$.
\begin{lemma} \label{deterr1}
Let $ 0\leq s\leq 1$ and $0\leq r\leq 2$ with $r+s\leq 2$. Then there holds
\begin{equation*}
\| A^{\frac{s}{2}}(\bar{E}(t) - \bar{E}_{h}(t) P_{h})\|
\leq c t^{\frac{r}{2}\alpha+\gamma-1} h^{2-s-r}.
\end{equation*}
\end{lemma}
\begin{proof}
Fix $g\in L^2(D)$. In the case $s=0$, by \eqref{eqn:solop} and \eqref{eqn:solop-semi}, there holds
\begin{align*}
 \|\big(\bar{E}(t)-\bar{E}_{h}(t)P_h\big)g \|& \leq \int_{\Gamma_{\theta,\delta}} e^{\Re{z}t}\|((z^{\alpha}
 +A)^{-1}-(z^\alpha+A_h)^{-1}P_h)g\||z|^{-\gamma} |\d z|.
\end{align*}
Since $\|((z^{\alpha} +A)^{-1} - (z^{\alpha} + A_{h})^{-1} P_{h} ) g \|
\leq ch^{2} \| g \|$ for all $z\in \Sigma_\theta$ (cf. \cite[p. 820]{FujitaSuzuki:1991} or
\cite[Lemma 3.4]{BazhlekovaJinLazarov:2015}), we have
\begin{equation*}
 \|( \bar{E}(t) - \bar{E}_{h}(t) P_{h}) g\| \leq c t^{\gamma -1} h^{2} \| g \|.
\end{equation*}
Meanwhile, by Lemmas \ref{lem:solop} and \ref{lem:solop-discrete} and the triangle inequality,
\begin{equation*}
 \|( \bar{E}(t) - \bar{E}_{h}(t) P_{h}) g\|\leq c t^{\alpha+\gamma -1} \| g \|.
\end{equation*}
Similarly, for $s=1$, there hold
\begin{equation*}
 \| A^{\frac{1}{2}}(  \bar{E}(t) - \bar{E}_h(t) P_h)g\| \leq c t^{\gamma -1} h \| g \|\quad\mbox{and}
 \quad \| A^{\frac{1}{2}}(  \bar{E}(t) - \bar{E}_h(t) P_h)g\|\leq ct^{\frac{\alpha}{2}+\gamma-1}\|g\|.
\end{equation*}
Now the desired assertion follows by interpolation.
\end{proof}

Now we analyze the temporal error of the approximation $B_nP_h$.
\begin{lemma}  \label{lem:Bn}
For $g \in H$, the function $V^n=B_nP_hg$ is given by {\rm(}with $\Gamma_{\theta,\delta}^\tau = \{ z \in \Gamma_{\theta,\delta}: \, |\Im z| \leq \frac{\pi}{\tau} \}${\rm)}
\begin{equation*}
  V^n = \frac{1}{2 \pi \rm i} \int_{\Gamma_{\theta,\delta}^\tau } e^{zt_{n}} (\tau^{-\alpha}\delta(e^{-z\tau})^{\alpha}
  +A_h)^{-1} \tau^\gamma\delta(e^{-z\tau})^{-\gamma}P_hg \, \d z.
\end{equation*}
\end{lemma}
\begin{proof}
Direct computation gives (with $V^0=0$)
\begin{equation*}
  \widetilde V(\zeta) = \sum_{n=0}^\infty V^n\zeta^n = \sum_{n=1}^\infty B_nP_hg\zeta^n = \Big(\sum_{n=1}^\infty B_n\zeta^n\Big)P_hg.
\end{equation*}
The defining relation \eqref{eqn:B} for $\widetilde B(\zeta)$ and $B_n$ leads to
\[
\widetilde{V} (\zeta) = ( \widetilde{B} (\zeta) -1) P_h g=  ( \tau^{-\alpha} \delta(\zeta)^{\alpha} +A_h  )^{-1}
\tau^{\gamma-1} \delta (\zeta)^{-\gamma}\zeta P_hg.
\]
By Cauchy integral formula, we have, for small $ \rho >0$:
\begin{align*}
V^{n}  = B_{n}P_h g = \frac{1}{2 \pi \rm i} \int_{|\zeta| = \rho} \zeta^{-n-1} \widetilde{V} (\zeta) \, \d \zeta.
\end{align*}
The assertion follows by the variable change $\zeta=e^{-z \tau}$ and then deforming $|\zeta|=\rho$ into $ \Gamma_{\theta,\delta}^\tau$.
\end{proof}

With Lemma \ref{lem:Bn} and the resolvent estimate on $A_h$, the following smoothing property and error
estimate on $B_n$ follow easily \cite{JinLazarovZhou:2016sisc}.
\begin{lemma} \label{lem:err-v-time}
For any  $s\in[0,1]$, there hold
\begin{equation*}
\|A_h^\frac{s}{2}B_n\| \leq t_{n+1}^{(1-\frac{s}{2})\alpha+\gamma-1}\quad\mbox{and}\quad \|A_h^{\frac{s}{2}} (\bar E_h(t_{n}) - B_n)P_h\| \leq c\tau \,  t_{n+1}^{(1-\frac{s}{2})\alpha+\gamma-2}.
\end{equation*}
\end{lemma}

For any $s\in [0,1]$, we define an index $s^*\equiv s^*(\alpha,\gamma)$ by
\begin{equation*}
  s^* = \left\{\begin{array}{ll}
    \infty, & \mbox{if } (1-\frac{s}{2})\alpha+\gamma -1\geq 0,\\
    \frac{2}{2-2(\alpha+\gamma)+s\alpha}, & \mbox{otherwise}.
  \end{array}\right.
\end{equation*}
For $s^*\geq2$, $s$ should satisfy the condition $s\leq 2-\frac{1-2\gamma}{\alpha}$.
Then the following property holds for $B_{n}$.

\begin{lemma} \label{lem_2_5}
For any $s\in[0,1]$ with $s<2-\frac{1-2\gamma}{\alpha}$ and $p \in [2, s^*)$, there hold
\begin{equation*}
\tau \sum_{j=1}^{n} \| B_{n-j} P_{h} \|_{\mathcal{L}_2^0}^p
\leq c\|A^{-\frac{s}{2}} \|_{\mathcal{L}_2^0}^p.
\end{equation*}
\end{lemma}
\begin{proof}
By Lemma \ref{lem:err-v-time}, we deduce
\begin{align*}
\tau \sum_{j=1}^{n} \| B_{n-j} P_{h} \|_{\mathcal{L}_2^0}^p
&\leq \tau \sum_{j=1}^{n} \|A_h^{\frac{s}{2}} B_{n-j}\|^p
\|A_h^{-\frac{s}{2}} P_hA^{\frac{s}{2}} \|^p
\|A^{-\frac{s}{2}} \|_{\mathcal{L}_2^0}^{p}\\
&\leq c\tau \sum_{j=1}^nt_{n-j+1}^{((1-\frac{s}{2})\alpha+\gamma -1)p} \| A^{-\frac{s}{2}} \|_{\mathcal{L}_2^0}^p
< \infty.
\end{align*}
where the second line follows from \eqref{eqn:AAh} and the choice of the exponent $p$.
\end{proof}

Last, we give an important error estimate. It is the main result of this section, and crucial
to both strong and weak convergence. Recall that $p'$ is the conjugate exponent of $p\geq 1$.
\begin{theorem}\label{thm:int}
For any $0\leq s \leq 1$, $p\in [1,s^*)$, there holds
\begin{align*}
\Big ( \sum_{j=0}^{n-1} \int_{t_{j}}^{t_{j+1}} \|A^{ \frac{s}{2}}
\big ( \bar E(t_{n}-t) -    B_{n-j} P_{h}  \big ) \|^p \, \d s \Big )^{1/p} \leq c(t_n^{\frac{r\alpha}{2}+\gamma-\frac1{p'}} h^{2-s-r} + t_n^{\max(\eta-1,0)} \tau^\mu),
\end{align*}
with $\eta=(1-\frac{s}{2})\alpha + \gamma -\frac1{p'}$ and the exponents $r$ and $\mu$ given respectively by
\begin{equation*}
   r \in\left\{\begin{array}{ll}
   (\frac{2}{\alpha}(p^{\prime-1}-\gamma),2-s], & p'\gamma<1,\\
   (0,2-s],  & p'\gamma =1,\\
   {[0,2-s]},  & p'\gamma > 1,
   \end{array}\right.
   \quad\mbox{and}\quad
  \mu = \left\{\begin{array}{ll}
   \eta, & \eta<1,\\
   1-\epsilon, & \eta=1,\\
   1, & \eta>1.
   \end{array}\right.
\end{equation*}
\end{theorem}
\begin{proof}
By the triangle inequality, we split the left hand side (LHS) into
\begin{align*}
{\rm LHS} &\leq \Big ( \sum_{j=0}^{n-1} \int_{t_j}^{t_{j+1}} \|A^\frac{s}{2}
(\bar E (t_n-t) - \bar E_h(t_n-t)P_h)\|^p \, \d t \Big )^{1/p} \\
& \quad +\Big (\sum_{j=0}^{n-1} \int_{t_j}^{t_{j+1}} \|A^\frac{s}{2}(\bar E_h(t_n-t)P_h- \bar E_h(t_n-t_j)P_h)\|^p\d t\Big )^{1/p} \\
&\quad + \Big (\sum_{j=0}^{n-1} \int_{t_j}^{t_{j+1}} \|A^\frac{s}{2}
(\bar E(t_n-t_j)-B_{n-j} P_h)\|^p \d t \Big )^{1/p}  := \sum_{i=1}^3 {\rm I}_i^{1/p}.
\end{align*}
It suffices to bound the three terms ${\rm I}_i$. By the choice of the exponent $r$, $ (\frac{r\alpha}{2}+\gamma -1)p > -1$, and
thus, by Lemma \ref{deterr1},
\begin{align*}
{\rm I}_1 
 &\leq c h^{(2-s-r)p} \int_{0}^{t_n} (t_n-t)^{(\frac{r\alpha}{2}+\gamma-1)p} \d t \\
 &\leq c t_n^{(\frac{r\alpha}{2}+\gamma-1)p+1} h^{(2-s-r)p}.
\end{align*}
For the second term ${\rm I}_2$, simple interpolation between $s=0,1$ allows replacing $A$ with $A_h$, and thus
\begin{align*}
{\rm I}_{2} \leq & \sum_{j=0}^{n-2} \int_{t_j}^{t_{j+1}} \|A_h^{\frac{s}{2}}
(\bar E_h(t_n-t)-\bar E_h(t_n-t_j))P_h\|^p\d t\\
& +  \int_{t_{n-1}}^{t_n} \|A_h^{\frac{s}{2}}(\bar E_h(t_n-t) -\bar E_h(\tau))P_h\|^p \d t
:= {\rm I}_{2,1} + {\rm I}_{2,2}.
\end{align*}
For the summation ${\rm I}_{2,1}$, by H\"older inequality and the smoothing property of $\bar{E}'_h(s)$,
\begin{align*}
{\rm I}_{2,1}&= \sum_{j=0}^{n-2} \int_{t_j}^{t_{j+1}} \|
   \int_{t_j}^s A_h^\frac{s}{2}\bar{E}'_h (t_{n}-t)P_h \d t \|^p \d s \\
& \leq \sum_{j=0}^{n-2} \int_{t_j}^{t_{j+1}} \tau^\frac{p}{p'}
  \int_{t_j}^s \|A_h^\frac{s}{2}\bar{E}'_h(t_n-t)P_h\|^p\d t\d s \\
& \leq c\tau^p\int_\tau^{t_n}\|A_h^{\frac{s}{2}}\bar{E}'_h(t)\|^p \d t
\leq c\tau^p\int_\tau^{t_n}t^{((1-\frac{s}{2})\alpha + \gamma-2)p} \d t.
\end{align*}
By the definition of $\eta$, $p((1-\frac{s}{2})\alpha + \gamma-2)=p(\eta-1)-1$, and then direct computation leads to
\begin{align*}
{\rm I}_{2,1} & \leq c\left\{\begin{array}{ll}
  \tau^{p\eta}, & \eta<1,\\
  \tau^{p}\ell_n, & \eta=1,\\
  \tau^{p}t_n^{p(\eta-1)}, & \eta>1,
\end{array}\right.
\end{align*}
with $\ell_n=\ln(1+t_n/\tau)$. For the term ${\rm I}_{2,2}$, by the triangle inequality and Lemmas \ref{lem:solop-discrete}
and \ref{lem:err-v-time}, we deduce
\begin{align*}
{\rm I}_{2,2} & \leq  c  \int_0^\tau \|A_h^\frac{s}{2}\bar E_h(t) \|^p \d t+c\int_0^\tau\|A_h^\frac{s}{2}\bar E_h(\tau)\|^p\d t  \\
& \leq c \int_0^\tau t^{  ((1-\frac{s}{2})\alpha + \gamma -1) p} \d t+ c \tau^{( (1-\frac{s}{2})\alpha + \gamma -1 )p+1}
\leq c \tau^{p\eta },
\end{align*}
where the last step holds due to the choice of the exponent $p\in[1,s^*)$.
For the third and last term ${\rm I}_3$, by Lemma \ref{lem:err-v-time}, there holds
\begin{align*}
{\rm I}_3 &= \sum_{j=0}^{n-1} \int_{t_j}^{t_{j+1}} \|A^{\frac{s}{2}}(\bar E_h(t_n-t_j)-B_{n-j})\|^p\, \d t\\
  & \leq c\tau^{p+1}\sum_{j=0}^{n-1}  (t_{n+1}- t_{j})^{((1-\frac{s}{2})\alpha+ \gamma -2)p}
  \leq c\left\{\begin{array}{ll}
      \tau^{p\eta}, & \eta<1,\\
  \tau^{p}\ell_n, & \eta=1,\\
  \tau^{p}t_n^{p(\eta-1)}, & \eta>1.
  \end{array}\right.
\end{align*}
Combining the preceding estimates on ${\rm I}_i$s completes the proof of the theorem.
\end{proof}

\begin{remark}
Note that for $p\in[1,s^*)$, $\eta>0$ and $\frac{2}{\alpha}(p^{\prime-1}-\gamma)<2-s$, and thus the condition on $r$
makes sense. The fractional orders $\alpha,\gamma$, the noise regularity index $s$, and the integrability index $p$
all enter into the final error estimate, and their properly balancing gives the best possible rate.
\end{remark}

\section{Strong and weak convergence}\label{sec:err-stochastic}
This part gives the strong and weak error estimates of the numerical approximation by the scheme \eqref{eqn:fullydiscrete}.
\subsection{Strong convergence}

Now we can state a strong convergence result in $L^{p}(\Omega;H)$ with $p \geq 2$.
\begin{theorem} \label{thm:strong}
Let $u (t_{n})$ and $ U^{n}$ be the solutions of problems \eqref{eqn:sfde} and \eqref{eqn:fullydiscrete}, respectively.
If $\|A^{-\frac s2}\|_{\mathcal{L}_0^2}<\infty$ for some $s\in[0,1]$ with $s < 2-\frac{1-2\gamma}{\alpha}$,
then for any $p \in [2,s^*)$ and $u_0\in L^p(\Omega;\dot H^q(\Omega))$, $0\leq q\leq 2$, there holds
\begin{align*}
  \|u(t_n)-U^n)\|_{L^p(\Omega; H)} \leq& c(\tau t_n^{-1+\frac{q}{2}\alpha}+h^2t_n^{\frac{q-2}{2}\alpha})\|u_0\|_{L^p(\Omega;\dot H^q(D))}
   + c(t_n^{\frac{r\alpha}{2}+\gamma-\frac12} h^{2-s-r} + t_n^{\max(\eta-1,0)} \tau^\mu).
\end{align*}
with $\eta=(1-\frac{s}{2})\alpha + \gamma-\frac{1}{2}$ and the exponents $r$ and $\mu$ given respectively by
\begin{equation*}
   r \in\left\{\begin{array}{ll}
   (\frac{2}{\alpha}(\frac12-\gamma),2-s], & \gamma<\frac12, \\
   (0,2-s],  & \gamma =\frac12,\\
   {[0,2-s]},  & \gamma > \frac{1}{2},
   \end{array}\right.\quad\mbox{and}\quad
  \mu = \left\{\begin{array}{ll}
   \eta, & \eta<1,\\
   1-\epsilon, & \eta =1,\\
   1, & \eta>1.
   \end{array}\right.
\end{equation*}
\end{theorem}
\begin{proof}
By the triangle inequality, we have
\begin{align*}
\| u(t_n) &- U^{n}\|_{L^{p}(\Omega;H)}
  \leq \|( \bar E(t_n) - B_nP_h) u_0 \|_{L^p(\Omega;H)} \\
 & +  \|\sum_{j=0}^{n-1}\int_{t_j}^{t_{j+1}}(\bar{E}(t_n-t)-B_{n-j}P_h)\d W(t)  \|_{L^p(\Omega;H)}:= {\rm I} + {\rm II}.
\end{align*}
In view of Lemma \ref{lem:error-homo}, it suffices to bound the term ${\rm II}$. By Burkholder's inequality
\eqref{burkholder0} and the condition $\|A^{-\frac{s}{2}}\|_{\mathcal{L}_0^2}<\infty$, there holds
\begin{align*}
{\rm II}^2 & \leq c \sum_{j=0}^{n-1} \int_{t_j}^{t_{j+1}} \| \bar E(t_n-t) - B_{n-j} P_h \|_{\mathcal{L}_2^0}^{2} \d t \\
& \leq c\| A^{-\frac{s}{2}} \|_{\mathcal{L}_2^0}^2
\sum_{j=0}^{n-1} \int_{t_j}^{t_{j+1}} \|A^\frac{s}{2}(\bar E(t_{n}-t)-B_{n-j} P_h) \|^{2} \d t.
\end{align*}
Then the desired assertion follows from Theorem \ref{thm:int} with $p=2$.
\end{proof}

\begin{remark}\label{rmk:strong}
The condition $s< 2-\frac{1-2\gamma}{\alpha}$ requests that the noise $W(t)$ should not be too rough, and the condition
always holds for trace class noise, since $\alpha+\gamma>1/2$, cf. \eqref{ass:alpha}. This restriction stems from the
limited smoothing property of the solution operator $\bar E(t)$, cf. Lemma \ref{lem:solop1}. For $u_0=0$ and
trace class noise, i.e., $s=0$, the following statements hold:
\begin{itemize}
  \item[$\rm(i)$] The spatial convergence rate is $O(h^{2-\frac{1-2\gamma}{\alpha}-\epsilon})$ for $\gamma<1/2$,
    and $O(h^2)$ for $\gamma>1/2$. The former is due to the limited smoothing property of
$\bar E(t)$, and it may be enhanced to $O(h^2)$ for smoother noise.
  \item[$\rm(ii)$] The temporal convergence rate is $O(\tau^{\min(1,\alpha+\gamma-\frac{1}{2}-\epsilon)})$.
  When $\gamma =1-\alpha$, it is $O(\tau^{\frac{1}{2}-\epsilon})$, which coincides with that for the stochastic heat
  equation \cite{Yan:2005}, but the spatial convergence rate is $O(h^2)$ only if $\alpha<1/2$ or the noise has extra regularity.
\end{itemize}
These convergence rates agree with the regularity results in Theorems \ref{thm:regularity} and
\ref{thm:time_regularity} in Appendix \ref{app:reg}.
\end{remark}

\subsection{Weak convergence}\label{sec:weak}
For the weak convergence,
first we give a Malliavin regularity of the solution to problem \eqref{eqn:sfde}.
\begin{proposition} \label{prop:Malliavin-reg}
If $\|A^{-\frac s2}\|_{\mathcal{L}_0^2}<\infty$ for some $s\in[0,1]$ with $s\leq 2-\frac{1-2\gamma}{\alpha}$,
then for any $p \geq 2$ and  $q \in [2, s^*)$, and for any $u_0\in L^p(\Omega;\dot H^q(\Omega))$, $0\leq q\leq 2$,
up to modification, there exists a unique stochastic process
$u: [0, T] \times \Omega \to H$ satisfying \eqref{eqn:mild} such that $u \in C([0, T]; M^{1,p, q}(H))$.
\end{proposition}
\begin{proof}
The proof is similar to \cite[Proposition 4.4]{AnderssonKovacs:2016}, and thus we only give a sketch.
First, we show $u \in L^{2}(0, T; M^{1, 2}(H))$. This can be done by first proving
$\| U^{n} \|_{M^{1, 2}(H)} + \| U^{n}\|_{M^{1, p, 2}} < \infty$,
by straightforward calculation of the term $D [U^{n}]  ( \sigma)$
(see \cite[Proposition 4.3]{AnderssonKovacs:2016}), and then proving the error estimate of
$\|u(t_{n}) - U^{n}\|_{L^{2}(\Omega; H)}$, by the argument of
\cite[Theorem  4.2]{AnderssonKovacs:2016}. Then a limiting procedure gives $u \in L^{2}(0, T; M^{1, 2}(H)).$

Since $u \in L^{2}(0,T;M^{1, 2}(H))$, we may apply  \cite[Proposition 3.5 (ii)]{FuhrmanTesssitore:2002} or
\cite[(3.8)]{AnderssonKruseLarsson:2016} to obtain the Malliavin derivative of the solution $u$: for  any $\sigma\in[0,T]$,
\begin{align*}
D [ u(t)] (\sigma)
&  = \left\{\begin{array}{ll} D [E(t) u_{0}] (\sigma)
+ D [ \int_{0}^{t} \bar{E} (t-s) \, \d W(s) ](\sigma), & \sigma \leq t \leq T,\\
0, & 0 <t <\sigma,
\end{array}\right.\\
&= \left\{\begin{array}{ll}  \bar{E} (t- \sigma), & \sigma \leq t \leq T,\\
0, & 0 <t <\sigma.
\end{array}\right.
\end{align*}
Then the smoothing property of $\bar{E}(t)$ in Lemma \ref{lem:solop} implies
\begin{align*}
& \| D[u(t)] \|_{L^p(\Omega;L^q(0, T;\mathcal{L}_2^0))}^q
 = \| D[u(t)] \|_{L^p(\Omega; L^q(0, t;\mathcal{L}_2^0))}^q \\
=&  \| \bar{E} (t- \cdot) \|_{L^q(0, t;\mathcal{L}_2^0)}^q
= \int_{0}^{t} \| \bar{E} (t-s) \|_{\mathcal{L}_2^0}^q \d s \\
 \leq& c \| A^{-\frac{s}{2}} \|_{\mathcal{L}_2^0}^q
\int_{0}^{t} s^{((1- \frac{s}{2}) \alpha + \gamma -1)q} \d s  < \infty,
\end{align*}
where the last inequality is due to the choice of the exponent $q$. This completes the proof.
\end{proof}

The next result gives a similar bound on the discrete solution  $U^{n}$.
\begin{proposition} \label{prop:Malliavin-reg-disc}
If $\|A^{-\frac s2}\|_{\mathcal{L}_0^2}<\infty$ for some $s\in[0,1]$ with $s\leq 2-\frac{1-2\gamma}{\alpha}$,
then for any $p \geq 2$ and  $q \in [2, s^*)$, and $u_0\in L^p(\Omega;\dot H^q(\Omega))$, $0\leq q\leq 2$,
the solution $U^{n}$ to \eqref{eqn:fullydiscrete} satisfies $\| U^{n} \|_{M^{1, p,q}(H)} \leq  c$.
 \end{proposition}
\begin{proof}
By the representation \eqref {eqn:fullydiscrete1}, we have
\begin{align*}
\| U^{n} \|_{L^{p}(\Omega;H)}
& \leq \|U_{h}^{n} \|_{L^{p}(\Omega;H)}
+ \Big \| \int_{0}^{T} \sum_{j=0}^{n-1} \chi_{[t_{j}, t_{j+1})} (s) B_{n-j} P_{h} \d W (s) \Big \|_{L^{p}(\Omega; H)} : = {\rm I}_{1} + {\rm I}_{2}.
\end{align*}
In view of Lemma \ref{lem:error-homo}, it suffices to bound the second term ${\rm I}_2 $. By
Burkholder's inequality \eqref{burkholder0} and Lemma  \ref{lem_2_5} with $p=2$, we get
\begin{align}
{\rm I}_{2}& \le c \| \sum_{j=0}^{n-1} \chi_{[t_{j}, t_{j+1})} (s) B_{n-j} P_{h} \|_{L^2(0, T;\mathcal{L}_2^0)} =  c\big(
\tau \sum_{j=0}^{n-1} \| B_{n-j} P_{h} \|_{\mathcal{L}_2^0}^2 \big )^{\frac{1}{2}}  < \infty. \notag
\end{align}
This directly implies $\| U^{n} \|_{L^{p}(\Omega;H)} \leq c$. 
Next we bound the Malliavin derivative  $D [U^{n}] (\sigma)$ of $U^{n}$, $\sigma \in [0, T]$. By applying the
Malliavin derivative to the representation \eqref {eqn:fullydiscrete1} termwise and noting the identity
$ D[\int_{t_{j}}^{t_{j+1}} B_{n-j}P_h \d W(s)] (\sigma) = \chi_{[t_{j}, t_{j+1})}(\sigma) B_{n-j}P_h$
\cite[Prop. 3.16]{AnderssonKruseLarsson:2016}, we obtain
$$
D[U^{n}] (\sigma)=\sum_{j=0}^{n-1} \chi_{[t_{j}, t_{j+1})}(\sigma) B_{n-j}P_h,\quad \sigma\in[0,T].
$$
Hence, by Lemma \ref{lem_2_5}, there holds
\begin{align*}
& \|D[U^{n}]\|_{L^p(\Omega;L^q(0,T;\mathcal{L}_2^0))}^q
 =\|\sum_{j=0}^{n-1}\chi_{[t_{j}, t_{j+1})}(s) B_{n-j}P_h\|_{L^p(\Omega;L^q(0,T;\mathcal{L}_2^0))}^q  \\
 =& \|\sum_{j=0}^{n-1} \chi_{[t_{j}, t_{j+1})}(s) B_{n-j} P_h\|_{L^q(0,T;\mathcal{L}_2^0)}^q
=\tau \sum_{j=0}^{n-1} \| B_{n-j}P_h\|_{\mathcal{L}_2^0}^q< \infty.
\end{align*}
This completes the proof of the proposition.
\end{proof}

Last, we can give the weak convergence of the approximation $U^n$.
\begin{theorem} \label{thm:weak}
Let  $u (t_{n})$ and $ U^{n}$ be the solutions of \eqref{eqn:sfde} and
\eqref{eqn:fullydiscrete}, respectively, and $\Phi \in \mathcal{G}_{p}^{2,2}(H;\mathbb{R})$.
If $\|A^{-\frac s2}\|_{\mathcal{L}_0^2}<\infty$ for some $s\in[0,1]$ with $s< 2-\frac{1-2\gamma}{\alpha}$,
then for any $p\in [2,s^*)$ and $u_0\in L^p(\Omega;\dot H^q(D))$, $0\leq q\leq 2$, there holds
\begin{align*}
  |\mathbb{E}[\Phi(u(t_n))-\Phi(U^n)]|  \leq  c(\tau t_n^{-1+\frac{q}{2}\alpha}+h^2t_n^{\frac{q-2}{2}\alpha})
  \|u_0\|_{L^p(\Omega;\dot H^q(D))}+c(t_n^{\frac{r\alpha}{2}+\gamma-\frac1p} h^{2-s-r} + t_n^{\max(\eta-1,0)} \tau^\mu),
\end{align*}
with $\eta= (1-\frac{s}{2})\alpha + \gamma - \frac{1}{p}$ and the exponents $r$ and $\mu$ given respectively by
\begin{equation*}
   r \in\left\{\begin{array}{ll}
   (\frac{2}{\alpha}(\frac1p-\gamma),2-s], &\gamma p<1, \\
   (0,2-s],  & \gamma p=1,\\
   {[0,2-s]},  & \gamma p> 1,
   \end{array}\right.\quad\mbox{and}\quad
  \mu = \left\{\begin{array}{ll}
   \eta, & \eta < 1,\\
   1-\epsilon, & \eta = 1,\\
   1, & \eta>1.
   \end{array}\right.
\end{equation*}
\end{theorem}
\begin{proof}
In view of the Gel'fand triple $M^{1, p} (H) \subset L^{2}(\Omega;H) \subset M^{1, p}(H)^{*}$, there holds
\begin{align*}
|\mathbb{E} [ \Phi (u(t_{n})) - \Phi (U^n)]|
 &= |\mathbb{E} [(\int_{0}^{1} \Phi^{\prime} (\rho u(t_{n}) + (1- \rho) U^{n} ) \, \d \rho, u(t_{n})- U^{n} ) ]| \\
& \leq  \|   \int_{0}^{1} \Phi^{\prime} (\rho u(t_{n}) + (1- \rho) U^{n} ) \, \d \rho \|_{M^{1, p} (H)}
 \| u(t_{n}) - U^{n} \|_{M^{1, p}(H)^{*}}.
\end{align*}
Now we claim any $p\in[2,s^*)$, $\|\int_0^1\Phi'(\rho u(t_n) + (1- \rho) U^n) \d \rho \|_{M^{1, p}(H)} < \infty$.
Actually, by Lemma \ref{lem:bd-Malliavin} with $\gamma = \Phi^{\prime}$ and $r=1$ and $ q=p$,   $p\in [2,s^*)$, we get
\begin{align*}
 &\| \Phi^{\prime} (\rho u(t_{n}) + (1- \rho) U^{n} )\|_{M^{1, p}(H)}\\
\leq& c(1+ \| \rho u(t_{n})  + (1- \rho) U^{n} \|_{M^{1, p}(H)})\\
\leq &c(1 + \| u(t_{n}) \|_{M^{1,2p, p}(H)} +  \| U^{n} \|_{M^{1,2p, p}(H)}).
\end{align*}
Thus the claim follows from Propositions \ref{prop:Malliavin-reg}  and \ref{prop:Malliavin-reg-disc}.
It remains to bound  $ \|u(t_n) - U^n\|_{M^{1, p}(H)^\ast}$. By the triangle inequality,
\begin{align*}
 \|u(t_n)-U^n\|_{M^{1, p}(H)^{*}} \leq & \|(E(t_n)-B_nP_h)u_0\|_{M^{1, p}(H)^{*}} \\
& + \| \int_{0}^{t_n} \bar{E}(t_n-t)\d W(t)
- \sum_{j=0}^{n-1}  \int_{t_j}^{t_{j+1}} B_{n-j}P_h \d W(t)\|_{M^{1, p}(H)^{*}} := {\rm I + II}.
\end{align*}
In view of Lemma \ref{lem:error-homo}, it suffices to bound ${\rm II}$. By Burkholder inequality \eqref{burkholder}, we have
\begin{align*}
{\rm II}^{p'} & \leq c\sum_{j=0}^{n-1} \int_{t_j}^{t_{j+1}} \| \bar E(t_n-t) - B_{n-j} P_h \|_{\mathcal{L}_2^0}^{p'} \d t\\
& \leq c\| A^{-\frac{s}{2}} \|_{\mathcal{L}_2^0}^{p'}\sum_{j=0}^{n-1} \int_{t_j}^{t_{j+1}} \|A^{\frac{s}{2}}
(\bar E(t_{n}-t)-B_{n-j} P_{h}) \|^{p'} \, \d s.
\end{align*}
Then Theorem \ref{thm:int} with $p'\in(1,2)$ completes  the proof.
\end{proof}
\begin{remark}\label{rmk:weak}
The condition $s< 2-\frac{1-2\gamma}{\alpha}$ ensures that $s^*> 2$ so that the choice $p\in[2,s^*)$
is valid. We specialize Theorem \ref{thm:weak} to $u_0=0$ and trace class noise $W(t)$, i.e., $s=0$,
and distinguish two cases for the weak error estimates: {\rm(}a{\rm)} $\alpha+\gamma\ge1$ and {\rm(}b{\rm)} $\alpha+\gamma< 1$:
\begin{itemize}
  \item[$\rm(a)$] The exponent $p$ can be arbitrarily large. Thus, the spatial convergence rate is $O(h^2)$
  for any $\gamma\geq1-\alpha$, and the temporal one $O(\tau^{\min(1,\alpha+\gamma-\epsilon)})$. When $\gamma =1-\alpha$, the temporal
  rate is $O(\tau^{1-\epsilon})$, which coincides with that for the stochastic heat equation, but the spatial rate is $O(h^2)$ only if
  $\alpha<1/2$ or $W(t)$ has extra regularity.
  \item[$\rm(b)$] The largest possible exponent $p$ is $p=\frac1{1-\alpha-\gamma}-\epsilon>2$. Hence, the spatial
  rate is $O(h^2)$ for $\gamma>\frac{1-\alpha}{2}$, and $O(h^{4-\frac{2(1-2\gamma)}{\alpha}-\epsilon})$ for
  $\gamma\leq\frac{1-\alpha}{2}$ {\rm(}note that $4-\frac{2}{\alpha}(1-2\gamma)\in (0,2]$ under the designated
  conditions \eqref{ass:alpha} and $\gamma\leq \frac{1-\alpha}{2}${\rm)}. The temporal rate is always $O(\tau^{\alpha+\gamma-\epsilon})$.
\end{itemize}
\end{remark}

\begin{remark}
Note that our analysis relies only on Laplace transform and resolvent estimate. Hence, it applies also to
slightly more general positive kernels, for which however we are not aware of any
mathematical modeling with fractionally integrated Gaussian noise.
\end{remark}
\section{Numerical experiments and discussions} \label{sec:numeric}

Now we present numerical results for the model \eqref{eqn:sfde}
with $0 < \alpha <  1$ and $ 0 \leq  \gamma \leq 1$ on the unit interval $D=(0,1)$
to illustrate the theoretical analysis.

\subsection{Implementation details}\label{ssec:implement}
First, we describe the implementation of the noise term $W(t)$, following
\cite{Yan:2005}. We consider only the case the covariance operator $Q$ shares
the eigenfunctions with the operator $A$. Recall the Fourier expansion of the
Wiener process $W(t)$ in \eqref{eqn:Wiener}:
\[
W( t) = \sum_{\ell=1}^{\infty} \gamma_{\ell}^\frac{1}{2} e_{\ell} {\beta}_{\ell}(t),
\]
where ${\beta}_{\ell}$, $\ell=1, 2, \dots,$ are i.i.d. Brownian motions,
and $\gamma_\ell$ and $e_\ell$ are the eigenvalues (ordered nondecreasingly, with multiplicity
counted) and eigenfunctions of $Q$. Thus the $L^2(D)$-projection $P_hW(t)\in X_h$ is given by
(with $L$ term truncation)
\[
 (P_hW(t),\chi) = \sum_{\ell=1}^\infty \gamma_\ell^\frac{1}{2}\beta_\ell(t)(e_\ell,\chi)
 \approx \sum_{\ell=1}^{L} \gamma_{\ell}^\frac{1}{2}\beta_{\ell} (t) (e_{\ell},\chi),    \quad \forall \chi\in X_h.
\]
Since $\beta_\ell(t)$s are i.i.d. Brownian motions, the increments $\Delta \beta_\ell^k$ are given by
\begin{equation*}
  \Delta \beta_\ell^k = \beta_\ell(t_k)-\beta_\ell(t_{k-1})\sim \sqrt{\tau}\mathcal{N}(0,1),\quad k=1,2,\ldots,N,
\end{equation*}
where $\mathcal{N}(0,1)$ denotes the standard Gaussian distribution.
Further, the fractionally integrated noise $P_h \dot W(t_k)$ is approximated by backward difference
\begin{equation*}
  P_h \dot W(t_k) \approx\frac{P_hW(t_{k}) - P_hW(t_{k-1})}{\tau},
\end{equation*}
and with $P_h\dot W(t_0)=0$. Using Gr\"{u}nwald-Letnikov formula \eqref{eqn:GL}, the
term ${_0I_t^\gamma} P_h\dot{W}(t_{n})$ is approximated by
\[
{_0I_t^\gamma} P_h\dot{W}(t_{n})\approx
\tau^{\gamma} \sum_{k=1}^{n} \beta_{n-k}^{(-\gamma)}
\Big [
\sum_{\ell=1}^{L} \gamma_{\ell}^\frac{1}{2}P_h e_{\ell}  \frac{\Delta \beta_{\ell}^k}{\tau}  \Big ].
\]

It is known that for a quasi-uniform triangulation $\mathcal{T}_h$, it is sufficient to take $L\geq N_h$ in the
truncation \cite{Yan:2005}, with $N_h$ being the FEM degree of freedom, in order to preserve the desired
convergence. The truncation number $L=N_h$ is employed in our numerical experiments.

Below we present numerical results on the unit interval $D=(0,1)$, and fix $u_0=0$.
The eigenfunctions $e_{\ell}(x)$ are given by $\sqrt{2}\sin(\ell \pi x)$, $\ell=1,2,\dots$,
and let $\gamma_{\ell}= \ell^{-m}, m \geq 0$. Thus, the borderline  for trace class noise is
$m=1$, and $m=0$ corresponds roughly to $s=-1$. The domain $D=(0,1)$ is divided into $M$ subintervals of length
$h=1/M$, and the time step size $\tau$ is fixed at $\tau=t/N$, where $t$ is the time of interest. To check
the convergence rate, we choose the $L^2(\Omega;H)$ norm for strong convergence, and $\Phi(u(t))=\int_D u(t)^2\d x$ for weak
convergence. All the expected values are computed with 100 trajectories.

\subsection{Numerical results for temporal convergence}

In this set of experiments, we fix the final time $t$ at $t=0.01$ and $M=100$. The reference solution is
computed with a much finer temporal mesh with $N=3200$. The numerical results for various
fractional orders $\alpha$ and $\gamma$ and trace class noise (with $m=2$) are given in Table
\ref{tab:gamma-odd-time-new}. In the table, the numbers in the bracket in the last column denote the
theoretical rates predicted by Theorems \ref{thm:strong} and \ref{thm:weak} (and Remarks \ref{rmk:strong}
and \ref{rmk:weak}), and for each $\alpha$ value, the first and second rows give the strong and weak
errors, respectively. When $s=0$, the theoretical rate is nearly $O(\tau^{\min(\alpha+\gamma-\frac{1}{2},1)})$ and $O
(\tau^{\min(\alpha+\gamma,1)})$ (up to possibly a logarithmic factor) in the strong and weak sense,
respectively. Overall, the empirical rates agree well with the theoretical ones. The convergence rate improves steadily
as the fractional orders $\alpha$ and $\gamma$ increase, due to the improved temporal solution regularity.
Further, note that the weak rate is generally not twice the strong one, unlike the case for the stochastic heat equation.

\begin{table}[htb!]
\centering
\caption{The $L^2(\Omega;H)$-error for trace class  noise $(m=2)$
at $t=0.01$.}\label{tab:gamma-odd-time-new}
\begin{tabular}{c|c|ccccc|c}
\hline
$\gamma$&$\alpha\backslash N$ &  40  &  80 & 160 & 320 &  640 & rate\\
\hline
0.3  & 0.2 & 6.68e-1 & 6.38e-1 & 6.07e-1 & 5.55e-1 & 4.98e-1 & 0.10 ($--$) \\
     &     & 3.82e-1 & 3.70e-1 & 3.26e-1 & 2.79e-1 & 2.23e-1 & 0.19 ($--$) \\
     & 0.4 & 1.19e-1 & 9.85e-2 & 8.14e-2 & 6.86e-2 & 5.83e-2 & 0.25 (0.20) \\
     &     & 1.76e-2 & 1.37e-2 & 7.58e-3 & 6.30e-3 & 3.37e-3 & 0.60 (0.70) \\
     & 0.6 & 1.05e-2 & 7.53e-3 & 5.13e-3 & 3.80e-3 & 2.58e-3 & 0.50 (0.40) \\
     &     & 2.65e-4 & 2.22e-4 & 3.71e-5 & 1.05e-5 & 1.49e-5 & 1.03 (0.90) \\
     & 0.8 & 1.22e-3 & 8.46e-4 & 5.61e-4 & 3.67e-4 & 2.36e-4 & 0.59 (0.60) \\
     &     & 2.64e-5 & 1.56e-5 & 7.95e-6 & 3.49e-6 & 3.31e-7 & 1.40 (1.00) \\
\hline
0.5  & 0.2 & 5.86e-2 & 5.10e-2 & 4.16e-2 & 3.34e-2 & 2.78e-2 & 0.26 (0.20) \\
     &     & 4.23e-3 & 3.36e-3 & 2.65e-3 & 1.48e-3 & 7.26e-4 & 0.63 (0.70) \\
     & 0.4 & 9.96e-3 & 7.16e-3 & 4.98e-3 & 3.81e-3 & 2.64e-3 & 0.47 (0.40) \\
     &     & 3.90e-4 & 2.33e-4 & 1.13e-4 & 5.75e-5 & 1.64e-5 & 1.14 (0.90) \\
     & 0.6 & 9.97e-4 & 6.73e-4 & 4.43e-4 & 2.97e-4 & 1.85e-4 & 0.60 (0.60) \\
     &     & 2.79e-5 & 1.33e-5 & 4.59e-6 & 3.19e-6 & 1.67e-6 & 1.01 (1.00) \\
     & 0.8 & 6.00e-4 & 3.90e-4 & 2.30e-4 & 1.37e-4 & 8.29e-5 & 0.71 (0.80) \\
     &     & 2.41e-6 & 1.64e-6 & 8.49e-7 & 2.74e-7 & 2.16e-7 & 0.87 (1.00) \\
\hline
0.7  & 0.2 & 4.95e-3 & 4.17e-3 & 3.21e-3 & 2.33e-3 & 1.75e-3 & 0.37 (0.40)\\
     &     & 5.45e-5 & 4.27e-5 & 2.53e-5 & 1.68e-5 & 2.11e-5 & 0.34 (0.90)\\
     & 0.4 & 4.39e-4 & 2.98e-4 & 2.15e-4 & 1.51e-4 & 1.08e-4 & 0.50 (0.60)\\
     &     & 5.50e-6 & 2.81e-6 & 1.32e-6 & 7.69e-7 & 4.36e-7 & 0.91 (1.00)\\
     & 0.6 & 4.39e-4 & 3.12e-4 & 2.06e-4 & 1.31e-4 & 7.20e-5 & 0.65 (0.80)\\
     &     & 1.90e-6 & 1.56e-6 & 5.88e-7 & 3.78e-7 & 2.03e-7 & 0.80 (1.00)\\
     & 0.8 & 2.44e-4 & 1.40e-4 & 6.75e-5 & 3.67e-5 & 1.93e-5 & 0.91 (1.00)\\
     &     & 4.36e-7 & 1.37e-7 & 1.16e-7 & 5.97e-8 & 1.83e-8 & 1.14 (1.00)\\
\hline
0.9  & 0.2 & 1.38e-4 & 7.94e-5 & 4.83e-5 & 2.69e-5 & 1.56e-5 & 0.78 (0.60)\\
     &     & 1.67e-6 & 8.28e-7 & 4.38e-7 & 2.35e-7 & 1.28e-7 & 0.92 (1.00)\\
     & 0.4 & 2.90e-4 & 2.00e-4 & 1.31e-4 & 8.44e-5 & 5.50e-5 & 0.59 (0.80)\\
     &     & 8.79e-7 & 2.83e-7 & 1.50e-7 & 3.12e-9 & 6.37e-8 & 0.94 (1.00)\\
     & 0.6 & 1.85e-4 & 1.03e-4 & 5.35e-5 & 3.07e-5 & 1.76e-5 & 0.84 (1.00)\\
     &     & 3.57e-7 & 1.53e-7 & 7.42e-8 & 4.94e-8 & 2.46e-8 & 0.96 (1.00)\\
     & 0.8 & 9.94e-5 & 5.28e-5 & 2.50e-5 & 1.28e-5 & 5.72e-6 & 1.02 (1.00)\\
     &     & 1.76e-7 & 1.12e-7 & 4.05e-8 & 1.50e-8 & 8.36e-9 & 1.09 (1.00)\\
\hline
\end{tabular}
\end{table}

By Theorems \ref{thm:strong} and \ref{thm:weak}, the regularity of $W(t)$
also affects the temporal convergence via the term $\|A^{-\frac{s}{2}}
\|_{\mathcal{L}_0^2}$: the convergence for white noise is slower than that for trace class
noise, cf. Table \ref{tab:m-time-new}. By the well-known asymptotics $O(j^2)$ of the 1D
negative Laplacian $A$, $m=0$ corresponds to roughly $s=1$, and thus Theorems \ref{thm:strong}
and \ref{thm:weak} yield the theoretical rates $O(\tau^{\min(\frac{\alpha}{2}+\gamma-\frac12,1)})$
and $O(\tau^{\min(\alpha+2\gamma-1,1)})$ in the strong and weak convergence, respectively; see Table
\ref{tab:m-time-new}. The empirical rates are slightly higher than the theoretical one.
Further, noise regularity (indicated by $m$) beyond trace class affects very little the temporal
convergence; see the results for $m=2,3$ in Table \ref{tab:m-time-new}.

\begin{table}[htb!]
\centering
\caption{The $L^2(\Omega;H)$-error at $t=0.01$ with $\gamma=0.4$ and
noise regularity index $m$.}\label{tab:m-time-new}
\begin{tabular}{c|c|ccccc|c}
\hline
$m$&$\alpha\backslash N$ &  40 & 80 & 160 & 320 & 640 & rate\\
\hline
0    & 0.3 & 1.13e-1 & 1.02e-1 & 8.91e-2 & 7.33e-2 & 6.06e-2 & 0.22 (0.05)\\
     &     & 8.12e-3 & 6.36e-3 & 6.71e-3 & 4.74e-3 & 5.50e-3 & 0.14 (0.10)\\
     & 0.5 & 2.21e-2 & 1.83e-2 & 1.48e-2 & 1.15e-2 & 8.41e-3 & 0.34 (0.15)\\
     &     & 1.09e-3 & 7.45e-4 & 4.78e-4 & 3.18e-4 & 1.94e-4 & 0.62 (0.30)\\
     & 0.7 & 4.47e-3 & 3.30e-3 & 2.36e-3 & 1.67e-3 & 1.11e-3 & 0.50 (0.25)\\
     &     & 8.24e-5 & 4.78e-5 & 3.19e-5 & 1.90e-5 & 8.09e-6 & 0.83 (0.50)\\
     & 0.9 & 1.66e-3 & 1.11e-3 & 7.17e-4 & 4.53e-4 & 2.75e-4 & 0.64 (0.30)\\
     &     & 7.49e-6 & 6.08e-6 & 2.11e-6 & 1.25e-6 & 6.04e-7 & 0.90 (0.70)\\
\hline
1    & 0.3 & 9.76e-2 & 8.46e-2 & 7.08e-2 & 6.23e-2 & 5.06e-2 & 0.23 (0.20)\\
     &     & 9.32e-3 & 7.89e-3 & 6.72e-3 & 4.52e-3 & 2.74e-3 & 0.44 (0.70)\\
     & 0.5 & 1.37e-2 & 1.03e-2 & 7.82e-3 & 5.79e-3 & 4.08e-3 & 0.43 (0.40)\\
     &     & 5.50e-4 & 3.66e-4 & 1.33e-4 & 7.44e-5 & 4.19e-5 & 0.92 (0.90)\\
     & 0.7 & 1.84e-3 & 1.23e-3 & 8.28e-4 & 5.44e-4 & 3.41e-4 & 0.60 (0.60)\\
     &     & 4.09e-5 & 1.90e-5 & 6.90e-6 & 2.82e-6 & 2.68e-6 & 0.98 (1.00)\\
     & 0.9 & 9.01e-4 & 5.50e-4 & 3.29e-4 & 2.03e-4 & 1.18e-4 & 0.73 (0.80)\\
     &     & 5.54e-6 & 2.37e-6 & 1.55e-6 & 5.69e-7 & 1.95e-7 & 1.20 (1.00)\\
\hline
2    & 0.3 & 9.04e-2 & 7.83e-2 & 6.85e-2 & 5.82e-2 & 4.48e-2 & 0.25 (0.20)\\
     &     & 7.91e-3 & 5.70e-3 & 2.47e-3 & 6.43e-4 & 7.08e-4 & 0.87 (0.70)\\
     & 0.5 & 1.10e-2 & 8.10e-3 & 5.76e-3 & 4.07e-3 & 2.82e-3 & 0.49 (0.40)\\
     &     & 2.70e-4 & 1.93e-4 & 1.02e-4 & 5.61e-5 & 1.82e-5 & 0.97 (0.90)\\
     & 0.7 & 1.03e-3 & 7.44e-4 & 5.11e-4 & 3.27e-4 & 2.01e-4 & 0.59 (0.60)\\
     &     & 1.49e-5 & 2.73e-6 & 1.70e-6 & 1.25e-6 & 8.04e-7 & 1.05 (1.00)\\
     & 0.9 & 7.15e-4 & 4.22e-4 & 2.60e-4 & 1.48e-4 & 8.82e-5 & 0.75 (0.80)\\
     &     & 4.34e-6 & 2.50e-6 & 1.34e-6 & 3.97e-7 & 2.89e-7 & 0.97 (1.00)\\
\hline
3    & 0.3 & 9.22e-2 & 7.53e-2 & 6.14e-2 & 5.44e-2 & 4.05e-2 & 0.29 (0.20)\\
     &     & 9.75e-3 & 6.80e-3 & 4.67e-3 & 2.39e-3 & 1.47e-3 & 0.68 (0.70)\\
     & 0.5 & 1.00e-2 & 6.86e-3 & 5.06e-3 & 3.26e-3 & 2.07e-3 & 0.57 (0.40)\\
     &     & 2.92e-4 & 1.17e-4 & 4.47e-5 & 2.53e-5 & 2.61e-5 & 0.87 (0.90)\\
     & 0.7 & 9.06e-4 & 6.16e-4 & 4.21e-4 & 2.85e-4 & 1.81e-4 & 0.57 (0.60)\\
     &     & 1.81e-5 & 1.04e-5 & 3.52e-6 & 2.28e-6 & 6.66e-7 & 1.19 (1.00)\\
     & 0.9 & 6.59e-4 & 3.99e-4 & 2.23e-4 & 1.38e-4 & 7.36e-5 & 0.79 (0.80)\\
     &     & 3.13e-6 & 1.08e-6 & 6.53e-7 & 5.00e-7 & 1.46e-7 & 1.10 (1.00)\\
\hline
\end{tabular}
\end{table}

\subsection{Numerical results for spatial convergence} Next we examine the spatial convergence. Here,
we fix the number $M$ of time steps at $M=200$ and the final time $t$ at $t=1$, and compute the
reference solution at $N=480$. The numerical results are given in Table \ref{tab:gamma-odd-space} for
trace class noise (with $m=2$) with various $\alpha$ and $\gamma$ values. A
convergence rate $O(h^2)$ is consistently observed for all combinations, concurring Theorems
\ref{thm:strong} and \ref{thm:weak}.

The influence of the noise regularity (indicated by $m$) on the convergence rates is shown in Table \ref{tab:m-space}.
It is observed that for $m=2$, the weak and strong rates saturate at $O(h^2)$, due to the use of
linear finite elements, despite the improved noise regularity. However, it deteriorates when the
noise regularity is lowered to the borderline of trace class (i.e., $m=1$) or white noise (i.e.,
$m=0$): for $m=1$, the strong and weak rates are predicted to be $O(h^{2-\frac{1-2\gamma}{\alpha}})$
and $O(h^2)$, respectively; and for $m=0$, they are $O(h^{1-\frac{1-2\gamma}{\alpha}})$ and
$O(h^{1-\min(1-2\gamma-\frac{\alpha}{2},0)})$, respectively. The empirical rates are much higher
than the theoretical ones when $m=0$, indicating an interesting superconvergence
phenomenon, whose precise mechanism remains to be ascertained.

\begin{table}[htb!]
\centering
\caption{The $L^2(\Omega;H)$-error with trace class noise ($m=2$)
at $t=1$.}\label{tab:gamma-odd-space}
\begin{tabular}{c|c|ccccc|c}
\hline
$\gamma$&$\alpha\backslash M$ &    10 &  20  &  40 & 80 & 160 &  rate\\
\hline
0.2  & 0.3 & 7.65e-3 & 2.01e-3 & 5.15e-4 & 1.27e-4 & 2.96e-5 & 2.00 ($--$)\\
     &     & 4.60e-3 & 1.17e-3 & 2.92e-4 & 7.17e-5 & 1.64e-5 & 2.03 ($--$)\\
     & 0.5 & 6.07e-3 & 1.63e-3 & 4.21e-4 & 1.05e-4 & 2.44e-5 & 1.98 (2.00)\\
     &     & 2.48e-3 & 6.33e-4 & 1.58e-4 & 3.88e-5 & 8.87e-6 & 2.03 (2.00)\\
     & 0.7 & 4.82e-3 & 1.32e-3 & 3.46e-4 & 8.74e-5 & 2.03e-5 & 1.97 (2.00)\\
     &     & 1.30e-3 & 3.35e-4 & 8.40e-5 & 2.06e-5 & 4.71e-6 & 2.02 (2.00)\\
     & 0.9 & 4.05e-3 & 1.12e-3 & 2.96e-4 & 7.52e-5 & 1.76e-5 & 1.96 (2.00)\\
     &     & 8.79e-4 & 2.25e-4 & 5.65e-5 & 1.38e-5 & 3.17e-6 & 2.02 (2.00)\\
\hline
0.6  & 0.3 & 2.39e-3 & 6.25e-4 & 1.59e-4 & 3.93e-5 & 9.09e-6 & 2.01 (2.00)\\
     &     & 4.68e-4 & 1.19e-4 & 2.97e-5 & 7.29e-6 & 1.66e-6 & 2.03 (2.00)\\
     & 0.5 & 2.30e-3 & 6.02e-4 & 1.53e-4 & 3.80e-5 & 8.78e-6 & 2.00 (2.00)\\
     &     & 4.22e-4 & 1.07e-4 & 2.68e-5 & 6.58e-6 & 1.50e-6 & 2.03 (2.00)\\
     & 0.7 & 2.26e-3 & 5.92e-4 & 1.50e-4 & 3.73e-5 & 8.64e-6 & 2.00 (2.00)\\
     &     & 4.02e-4 & 1.02e-4 & 2.56e-5 & 6.27e-6 & 1.43e-6 & 2.03 (2.00)\\
     & 0.9 & 2.27e-3 & 5.94e-4 & 1.51e-4 & 3.75e-5 & 8.67e-6 & 2.00 (2.00)\\
     &     & 4.09e-4 & 1.04e-4 & 2.60e-5 & 6.37e-6 & 1.45e-6 & 2.03 (2.00)\\
\hline
\end{tabular}
\end{table}

\begin{table}[htb!]
\centering
\caption{The $L^2(\Omega;H)$-error at $t=1$ with $\gamma=0.4$, and noise regularity index $m$.}\label{tab:m-space}
\begin{tabular}{c|c|ccccc|c}
\hline
$m$&$\alpha\backslash M$ &    10 &  20  &  40 & 80 & 160 &  rate\\
\hline
0    & 0.3 & 9.58e-3 & 3.48e-3 & 1.25e-3 & 4.36e-4 & 1.38e-4 & 1.52 (0.33)\\
     &     & 1.62e-3 & 4.44e-4 & 1.15e-4 & 2.89e-5 & 6.69e-6 & 1.98 (0.95)\\
     & 0.5 & 9.20e-3 & 3.40e-3 & 1.23e-3 & 4.33e-4 & 1.37e-4 & 1.51 (0.60)\\
     &     & 1.26e-3 & 3.51e-4 & 9.22e-5 & 2.32e-5 & 5.38e-6 & 1.96 (1.00)\\
     & 0.7 & 8.75e-3 & 3.30e-3 & 1.21e-3 & 4.29e-4 & 1.36e-4 & 1.49 (0.72)\\
     &     & 1.02e-3 & 2.87e-4 & 7.60e-5 & 1.92e-5 & 4.47e-6 & 1.95 (1.00)\\
     & 0.9 & 8.27e-3 & 3.17e-3 & 1.19e-3 & 4.24e-4 & 1.36e-4 & 1.48 (0.78)\\
     &     & 9.13e-4 & 2.56e-4 & 6.81e-5 & 1.72e-5 & 4.02e-6 & 1.95 (1.00)\\
\hline
1    & 0.3 & 5.11e-3 & 1.48e-3 & 4.16e-4 & 1.12e-4 & 2.79e-5 & 1.87 (1.33)\\
     &     & 1.20e-3 & 3.11e-4 & 7.83e-5 & 1.92e-5 & 4.40e-6 & 2.02 (2.00)\\
     & 0.5 & 4.70e-3 & 1.38e-3 & 3.95e-4 & 1.07e-4 & 2.69e-5 & 1.86 (1.60)\\
     &     & 8.84e-4 & 2.29e-4 & 5.77e-5 & 1.41e-5 & 3.24e-6 & 2.02 (2.00)\\
     & 0.7 & 4.34e-3 & 1.30e-3 & 3.75e-4 & 1.03e-4 & 2.59e-5 & 1.84 (1.72)\\
     &     & 6.90e-4 & 1.79e-4 & 4.53e-5 & 1.11e-5 & 2.55e-6 & 2.01 (2.00)\\
     & 0.9 & 4.09e-3 & 1.23e-3 & 3.59e-4 & 9.95e-5 & 2.51e-5 & 1.83 (1.78)\\
     &     & 6.27e-4 & 1.63e-4 & 4.11e-5 & 1.01e-5 & 2.31e-6 & 2.02 (2.00)\\
\hline
2    & 0.3 & 3.62e-3 & 9.49e-4 & 2.41e-4 & 6.00e-5 & 1.38e-5 & 2.00 (2.00)\\
     &     & 1.06e-3 & 2.71e-4 & 6.79e-5 & 1.66e-5 & 3.80e-6 & 2.03 (2.00)\\
     & 0.5 & 3.17e-3 & 8.39e-4 & 2.15e-4 & 5.35e-5 & 1.23e-5 & 2.00 (2.00)\\
     &     & 7.59e-4 & 1.93e-4 & 4.83e-5 & 1.18e-5 & 2.70e-6 & 2.03 (2.00)\\
     & 0.7 & 2.87e-3 & 7.63e-4 & 1.96e-4 & 4.89e-5 & 1.13e-5 & 1.99 (2.00)\\
     &     & 5.83e-4 & 1.48e-4 & 3.72e-5 & 9.12e-6 & 2.08e-6 & 2.03 (2.00)\\
     & 0.9 & 2.74e-3 & 7.28e-4 & 1.87e-4 & 4.67e-5 & 1.08e-5 & 1.99 (2.00)\\
     &     & 5.35e-4 & 1.36e-4 & 3.41e-5 & 8.36e-6 & 1.91e-6 & 2.03 (2.00)\\
\hline
\end{tabular}
\end{table}

In summary, all the numerical results indicate that the convergence rates are nearly sharp.
However, the rate in either strong or weak norm is limited to $O(h^2)$ and $O(\tau)$, and
it is of much interest to design high-order numerical schemes (in strong and / or weak
sense). The high-order convergence are expected from high regularity for $\alpha+\gamma>1$
(or $\alpha+\gamma>3/2$).


\bibliographystyle{abbrv}
\bibliography{frac}

\appendix

\section{Regularity theory}\label{app:reg}
In this appendix, we describe some regularity results for problem \eqref{eqn:sfde}.
First, we state a result on $\bar E(t)$.
\begin{lemma} \label{lem:solop1}
Let condition \eqref{ass:alpha} hold, and let {\rm(}with $\epsilon>0$ small{\rm)}
\begin{equation}\label{eqn:kappa}
  \kappa = \left\{\begin{array}{ll}
       2, &\quad \mbox{ if }1/2<\gamma\leq 1,\\
      2-\epsilon,  &\quad \mbox{ if } \gamma =1/2,\\
      2-\frac{1-2\gamma}{\alpha},     &\quad \mbox{ if } 0\leq \gamma <1/2.
  \end{array}\right.
\end{equation}
Then there holds
\begin{equation*}
    \|  \bar{E} (s)\|_{L^2(0,t;\dot H^ {\kappa}(D))} \leq ct^{ (2-\kappa)\frac{\alpha}{2}+\gamma - \frac{1}{2}}.
\end{equation*}
\end{lemma}
\begin{proof}
By Lemma \ref{lem:solop}, for any $\kappa\in[0,2]$, there holds
$\|A^\frac{\kappa}{2}\bar E(t)\| \leq ct^{(1-\frac\kappa2)\alpha+\gamma-1}$, and consequently
\begin{align*}
  \int_{0}^{t}  \| A^\frac{\kappa}{2} \bar E(s) \|^2 \d s & \leq c
   \int_{0}^{t}s^{(2-\kappa)\alpha+2\gamma-2} \d s.
\end{align*}
Under Assumption \ref{ass:alpha} and the choice of the exponent $\kappa$ in \eqref{eqn:kappa},
$(2-\kappa)\alpha+2\gamma-2>-1$ (except for the case $0\leq \gamma<1/2$ and $\kappa=2-\frac{1-2
\gamma}{\alpha}$)), thus we obtain the desired result upon integration. In the exceptional
case, $(2-\kappa)\alpha+2\gamma-2=-1$, and the assertion follows from direct computation.
\end{proof}

\begin{remark}
For $\gamma\in(1/2,1]$, $\bar E(t)$ has a maximum order two smoothing in space. However, for $\gamma\in[0,1/2)$,
$\kappa$ is restricted to $[0,2-\frac{1-2\gamma}{\alpha}]$. This again reflects a certain limited
smoothing property of $\bar E(t)$, and also restricts the best possible strong and weak spatial convergence rates.
\end{remark}

Now we can state the spatial regularity of the mild solution \eqref{eqn:mild}.
\begin{theorem} \label{thm:regularity}
Let condition \eqref{ass:alpha} hold, and $\kappa$ be defined by \eqref{eqn:kappa},
and $r,q\in\mathbb{R}$ with $0\leq r-q\leq 2$. For $u_{0}\in L^p(\Omega;
\dot H^q(D))$, let $u(t)$ be  the mild  solution  of problem \eqref{eqn:sfde}
defined in \eqref{eqn:mild}. Then there holds
\begin{equation*}
\|u(t) \|_{L^p(\Omega; \dot{H}^r(D))} \leq
ct^{-\alpha \frac{r-q}{2}} \| u_{0} \|_{L^p( \Omega;\dot H^q(D))}
+ c t^{(1-\frac{\kappa}{2})\alpha+\gamma-\frac12}  \| A^\frac{r-\kappa}{2}\|_{\mathcal{L}_2^0}.
\end{equation*}
\end{theorem}
\begin{proof}
By Lemma \ref{lem:solop}, it suffices to bound the stochastic integral.
By Burkholder's inequality \eqref{burkholder0},
\begin{equation*}
   \begin{aligned}
    \Big(\ee \| \int_{0}^{t} A^\frac{r}{2} \bar E (t-s) \,\d W(s) \|^p\Big)^{2/p} &\leq c \int_{0}^{t} \| A^\frac{r}{2}\bar E (s)  \|_{\mathcal{L}_2^0}^2 \, \d s
     \leq \|A^\frac{r-\kappa}{2}\|_{\mathcal{L}_2^0} \int_{0}^{t}  \|A^\frac{\kappa}{2}\bar E (s) \|  \d s,
   \end{aligned}
\end{equation*}
Then Lemma  \ref{lem:solop1}, the representation \eqref{eqn:mild} and the triangle inequality complete the proof.
 \end{proof}

To study the temporal regularity of the mild solution in \eqref{eqn:mild}, we need an elementary inequality.
\begin{lemma}\label{lem:est-int}
For $0\leq t_1<t_2$ and $\theta\in(1/2,3/2)$, then with $c=(3-2\theta)^{-\frac{1}{2}}(\theta-\frac{1}{2})^{-1}$, there holds
$
  \int_{t_1}^{t_2} \big(\int_0^{t_1}(t-s)^{2(\theta-2)}\d s\big)^\frac{1}{2}\d t \leq c(t_2-t_1)^{\theta-\frac{1}{2}}.
$
\end{lemma}
\begin{proof}
Since $\theta\in(1/2,3/2)$, i.e., $3-2\theta>0$, straightforward computation gives for  $t> t_1$
$\int_0^{t_1}(t-s)^{2(\theta-2)} \d s \leq (3-2\theta)^{-1}(t-t_1)^{2\theta-3}$. Thus simple
computation gives
$\int_{t_1}^{t_2} \Big(\int_0^{t_1}(t-s)^{2(\theta-2)}\d s\Big)^\frac{1}{2}\d t
\leq c(t_2-t_1)^{\theta-\frac{1}{2}}$, completing the proof.
\end{proof}

Now we can state the temporal H\"{o}lder regularity of the mild solution in \eqref{eqn:mild}.
\begin{theorem} \label{thm:time_regularity}
Let  condition \eqref{ass:alpha} hold, and $\kappa$ be defined in \eqref{eqn:kappa}. Let $u(t)$ be defined
in \eqref{eqn:mild}. Let $q,r\in [0,2]$ with $0\leq r-q \leq 2$, and $s\in[0,1]$ with $\max(0,
\frac{2(\alpha+\gamma)-3}{\alpha}+\epsilon)\leq r+s\leq \kappa$. Then for any $0<t_1<t_2<T$ and $p\geq 2$
and $u_0\in L^p(\Omega;\dot H^q(D))$, there holds
\begin{equation*}
\|u(t_2)-u(t_1)\|_{L^p(\Omega;\dot{H}^r(D))} \leq ct_1^{-(1+\alpha\frac{r-q}{2})}(t_2-t_1)\| u_0\|_{L^p(\Omega;\dot{H}^q(D))} +
 c (t_2-t_1)^{(1-\frac{r+s}{2})\alpha + \gamma -\frac{1}{2}}\|A^{-\frac{s}{2}}\|_{\mathcal{L}_2^0}.
\end{equation*}
\end{theorem}
\begin{proof}
By the representation \eqref{eqn:mild}, we have the splitting
\begin{equation*}
\begin{aligned}
u(t_{2})  - u (t_{1}) =
 & (E(t_{2}) u_{0} - E (t_{1}) u_{0}) + \int_{0}^{t_{1}} (
\bar{E} (t_{2} -s) - \bar{E} (t_{1} -s) ) \d W(s) \\
&+ \int_{t_{1}}^{t_{2}} \bar{E} (t_{2} - s) \d W(s):= \mathrm{I} + \mathrm{II} + \mathrm{III}.
\end{aligned}
\end{equation*}
Next we bound the three terms separately. The first term can be bounded directly by Lemma \ref{lem:solop}.
Next, for $\alpha+\gamma\neq1$ (the case $\alpha+\gamma=1$ is similar), by
stochastic Fubini theorem \cite[Theorem 4.33]{DaPratoZabczyk:2014},
\begin{equation*}
  \begin{aligned}
    \mathrm{II} &= \int_0^{t_1}(\bar{E}(t_2-s)-\bar E(t_1-s))\d W(s)
        = \int_{t_1}^{t_2}\int_0^{t_1}\bar E'(t-s)\d W(s)\d t.
  \end{aligned}
\end{equation*}
Thus by Burkholder's inequality \eqref{burkholder0} and Lemma \ref{lem:solop}, we have (with $\theta=\alpha+\gamma-\frac{r+s}{2}\alpha$)
\begin{equation*}
  \begin{aligned}
    \|\mathrm{II}\|_{L^p(\Omega;\dot H^r(D))}
     &\leq  \int_{t_1}^{t_2}\|\int_0^{t_1}A^\frac{r}{2}\bar E'(t-s)\d W(s)\|_{L^p(\Omega;H)}\d t\\
     &\leq c\int_{t_1}^{t_2}\Big(\int_0^{t_1}\|A^\frac{r+s}{2} \bar E'(t-s)\|\|A^{-\frac{s}{2}} \|_{\mathcal{L}_2^0}\Big)^2\d s\Big)^\frac{1}{2}\d t\\
     &=c\int_{t_1}^{t_2}\Big(\int_0^{t_1} (t-s)^{2(\theta-2)}ds\Big)^\frac{1}{2}\d t\|A^{-\frac s{2}} \|_{\mathcal{L}_2^0}.
  \end{aligned}
\end{equation*}
For $1/2<\alpha+\gamma<3/2$, the
condition on $r+s$ ensures $\theta\in(1/2,3/2)$. Similarly, for $\alpha+\gamma\geq 3/2$ and
$r\in(\frac{2(\alpha+\gamma)-3}{\alpha},2]$, we have
also $\theta\in(1/2,3/2)$. Thus, Lemma \ref{lem:est-int} implies
\begin{equation*}
  \|\mathrm{II}\|_{L^p(\Omega;\dot H^r(D))} \leq c(t_2-t_1)^{\theta -\frac{1}{2}} \|A^{-\frac s2} \|_{\mathcal{L}_2^0}.
\end{equation*}
Last, by Burkholder's inequality \eqref{burkholder0} and Lemma \ref{lem:solop} with $p=r+s$ and $q=0$, we deduce
\begin{equation*}
\begin{aligned}
\|\mathrm{III}\|_{L^p(\Omega;\dot{H}^r(D))}^2&\leq c\int_{t_{1}}^{t_{2}}\|A^{\frac{r+s}{2}}\bar{E} (t_{2}-s) A^{-\frac{s}{2}}\|_{\mathcal{L}_2^0}^2\d s\\
& \leq c\|A^{-\frac{s}{2}}\|_{\mathcal{L}_2^0}^2 \int_{t_1}^{t_2}(t_{2}-s)^{2\theta-2} \, \d s
 = c(t_2-t_1)^{2\theta-1}\|A^{-\frac{s}{2}}\|_{\mathcal{L}_2^0}^2.
\end{aligned}
\end{equation*}
Combining these estimates together completes the proof of the theorem.
\end{proof}

\begin{remark}
The condition $\max(0,\frac{2(\alpha+\gamma)-3}{\alpha}+\epsilon)\leq r+s\leq \kappa$ is only for
$\alpha+\gamma\geq3/2$ and restricts the discussion to H\"{o}lder continuity in time.
For $u_0=0$ and trace class noise, i.e., $s=0$, $r=0$,
\begin{equation*}
\|u(t_2)-u(t_1) \|_{L^{2} (\Omega;H)} \leq c (t_2-t_1)^{\alpha+\gamma-\frac{1}{2}}   \| Q^\frac{1}{2} \|_{HS}.
\end{equation*}
The case of $\alpha=1$ and $\gamma=0$ recovers the well known regularity result of
stochastic heat equation.
\end{remark}
\end{document}